\documentclass[a4paper, 11pt, twoside]{article}
\usepackage[english]{babel}
\usepackage[T1]{fontenc}
\usepackage[utf8]{inputenc}
\usepackage{fancyhdr}
\usepackage{float}
\usepackage{graphicx}
\usepackage{amsfonts}
\usepackage{amsthm}
\usepackage{amsmath}
\usepackage{mathtools}
\usepackage{hyperref}
\usepackage{a4wide}
\usepackage{mathrsfs}
\usepackage{enumerate}
\usepackage{amssymb}

\numberwithin{equation}{section}

\newtheorem{thm}{Theorem}[section]
\newtheorem{def1}[thm]{Definition}
\newtheorem{lemma}[thm]{Lemma}
\newtheorem{cor}[thm]{Corollary}
\newtheorem{prop}[thm]{Proposition}
\newtheorem{remark}[thm]{Remark}
\newtheorem{assumption}[thm]{Assumption}

\DeclareMathOperator{\dist}{dist}

\DeclareMathOperator{\Per}{Per}
\DeclareMathOperator{\supp}{supp}
\DeclareMathOperator{\divv}{div}

\DeclareMathOperator{\tr}{tr}
\newcommand{\R}{\mathbb{R}}
\newcommand{\N}{\mathbb{N}}

\newcommand{\m}{\mathfrak{m}}
\newcommand{\di}{\mathsf{d}}

\DeclareMathOperator{\Test}{Test}
\DeclareMathOperator{\Hess}{Hess}

\DeclareMathOperator{\sink}{sn}

 \newcommand{\nchi}{{\raise.3ex\hbox{$\chi$}}}
 \newcommand{\eps}{\varepsilon}
 \newcommand{\bd}{\bold\Delta}
 \newcommand{\sfd}{{\sf d}}
 \newcommand{\X}{{\rm X}}
 \newcommand{\mm}{\m}
 \newcommand{\Fl}{{\sf Fl}}
 \newcommand{\RCD}{{\sf RCD}}
 \newcommand{\pr}{{\mathscr P}}
 \newcommand{\e}{{\rm e}}
 \renewcommand{\b}{{\rm b}}
 \renewcommand{\Pr}{{\rm Pr}}
 \renewcommand{\d}{{\rm d}}
 \newcommand{\restr}[1]{\lower3pt\hbox{$|_{#1}$}}
 \newcommand{\Lip}{{\rm Lip}}
 \newcommand{\Id}{{\rm Id}}
 \newcommand{\la}{\langle}
 \newcommand{\ra}{\rangle}
 \newcommand{\ms}{{\sf ms}}
 \renewcommand{\div}{{\rm div}}
 \newcommand{\sfT}{{\sf T}}
 \newcommand{\para}{\parallel}
 \newcommand{\fr}{\hfill$\blacksquare$}   

\newcommand{\ppi}{{\mbox{\boldmath$\pi$}}}

\title{A general splitting principle on $\RCD$ spaces and applications to spaces with positive spectrum} 
 
\begin{document}

\author{Nicola Gigli\ \thanks{SISSA, ngigli@sissa.it}  \and
	Fabio Marconi
	\thanks{SISSA, fmarconi@sissa.it}}

\maketitle

\begin{abstract}
In this paper we develop a general `analytic' splitting principle for $\RCD$ spaces: we show that if there is a function with suitable Laplacian and Hessian, then the space is (isomorphic to) a warped product.

Our result covers most of the splitting-like results currently available in the literature about $\RCD$ spaces. We then apply it to extend to the non-smooth category some structural property of Riemannian manifolds obtained by Li and Wang.
\end{abstract}
\tableofcontents

\section{Introduction}
A central theme in geometric analysis is the understanding of how curvature affects the shape of the space under consideration. Often, studies in this direction come in the form of suitable geometric/analytic inequalities being valid in spaces satisfying appropriate curvature bounds. When this occurs, it becomes interesting to study the equality case, which typically comes with strong rigidity of the underlying geometry: it is not infrequent - and in fact very common when dealing with lower Ricci bounds - that in such rigid situations a (warped) product structure emerges. This is the case, for instance, in the two prototypical examples of rigidity in connection with lower Ricci bounds: the Cheeger-Gromoll splitting theorem and the volume-cone-to-metric-cone principle (see \cite{Cheegersurvey} and references therein). Almost by definition, such product structures emerge when one can find a smooth function $\b:M\to\R$ satisfying
\begin{equation}
\label{eq:introb}
\begin{split}
|\d \b|&\equiv 1,\\
\Delta\b&=\psi_\mm\circ\b,\\
\Hess\b&=\psi_\sfd\circ\b(\Id-\e_1\otimes\e_1)\qquad\text{where }\e_1:=\tfrac{\nabla\b}{|\nabla\b|},
\end{split}
\end{equation}
for suitable functions $\psi_\mm,\psi_\sfd:M\to\R$. In this  case it is easy to see that $M\sim \R\times_w N$ where:
\begin{itemize}
\item[-] The smooth manifold $N$ is given by $N:=\b^{-1}(\{0\})$,
\item[-] The metric-measure isomorphism sends $x\in M$ to $(\b(x),\Pr(x))\in\R\times\N$, where $\Pr(x)$ is the point along the gradient flow trajectory $\gamma'_t=\nabla\b(\gamma_t)$, $\gamma_0=x$ such that $\Pr(x)\in N$.
\item[-] The metric tensor on $\R\times_wN$ is given by $\d t^2+w_\sfd(t)(\d x')^2$, where $(\d x')^2$ is the metric on $N$ and
\[
w_\sfd(t):=\exp\Big(\int_0^t\psi_\sfd(s)\,\d s\Big).
\]
\item[-] The measure on $\R\times_wN$ is given by $\d \mathcal L^1(t)\otimes w_\mm(t) \d{\rm vol}_N(x')$, where ${\rm vol}_N$ is the volume measure on $N$ and
\[
w_\mm (t):=\exp\Big(\int_0^t\psi_\mm(s)\,\d s\Big).
\]
\end{itemize}
For instance, the splitting theorem corresponds to the above with $\psi_\mm=\psi_\sfd=0$ and the volume-cone-to-metric-cone theorem corresponds to $\psi_\sfd=\frac{1}{t}$ and $\psi_\mm=\frac{{\rm dim}(M)-1}{t}$.

\medskip

Starting from these considerations, we develop a general splitting principle in the framework of $\RCD(K,N)$ spaces, that are non-smooth metric measure structures having, in a suitable weak sense, Ricci curvature bounded from below by $K$ and dimension bounded from above by $N$. These structures have been introduced in \cite{Gigli12}, after the studies in \cite{Sturm06I}, \cite{Sturm06II}, \cite{Lott-Villani09}, \cite{AmbrosioGigliSavare11-2}. We refer to the surveys \cite{Villani2017},\cite{AmbrosioICM}, \cite{DGG} for more about the subject and detailed bibliography.

The first result we prove, see Theorem \ref{thm:genspl}, can be used to recover the rigidity obtained in \cite{Gigli13}, \cite{DPG16}, \cite{CDPSW21}, meaning that in obtaining the rigidity results of these papers, once   one has the function $\b$ as in \eqref{eq:introb}, then the corresponding product structure follows. In saying this we remark that in general a key  difficulty in proving that a product structure exists is in finding a function $\b$ as in   \eqref{eq:introb}. In this sense, our  splitting principle can be seen as a general tool that allows to translate the analytic information encoded in  \eqref{eq:introb} into a geometric information. Speaking of technical tools, many, but not all, of those that we use here come from \cite{DPG16}. A main difference, also with respect to the more recent  \cite{CDPSW21}, is that in studying how the gradient flow of $\b$ acts on vector fields, we have to (more) carefully distinguish between the component parallel to $\nabla \b$ and the one orthogonal to it: in the previous studies that we mentioned the specific geometry of the problem provided some simplification that is not present here, see in particular Section \ref{se:flowdist}.

In principle, one would like to obtain sharp informations about the $\RCD$ property satisfied by the quotient space $N$ in the above. This, however, seems tricky to do in our generality and our results in this direction are sub-optimal: we prove, under quite general assumptions on the warping functions, that if the original space is $\RCD(K,N)$, then so is the quotient one. In particular, and in line with \cite{CDPSW21}, we do not see an improving in the upper dimension bound. It seems that to obtain this a more careful analysis of the Bochner inequality - akin to that in \cite{Ketterer13} - on warped spaces is needed, but this is outside the scope of this paper.

\medskip

In the second part of the paper we generalize a result proved by Li and Wang in the smooth category in \cite{LW01} 
to the framework of $\RCD$ spaces. The statement   provides geometric rigidity under assumptions that combine informations on the first eigenvalue of the Laplacian and the behaviour of the space at infinity:
\begin{thm}\label{mainthm}
Let \((\X,\di,\m)\) be an \(\RCD(-(N-1),N)\) space with \(N\geq 3\) and assume that the first eigenvalue of the Laplacian \(\lambda_1\) is \(\geq N-2\). Then one of the following holds:
\begin{itemize}
	\item[i)] \(\X\) has only one end with infinite volume;
	\item[ii)] \(\X\) is isomorphic as metric measure space to a warped product space \(\R\times_w\X'\), where $\X'$ is a compact \(\RCD(-(N-1),N)\) space and the warping functions are
	\[
	w_\sfd(t):=\cosh(t),\qquad\text{ and }\qquad w_\mm(t):=\cosh^{N-1}(t).
	\] 
	Moreover, in this case \(\lambda_1=N-2\).
\end{itemize}
\end{thm}
See Section \ref{chvolends} for the definitions of `first eigenvalue of the Laplacian' and `end', while we refer to Section \ref{se:setup} for the definition of warped product  \(\R\times_w\X'\).

Thanks to the general splitting Theorem \ref{thm:genspl} proved before, in order to obtain this result it suffices to produce  a function $\b$ as in \eqref{eq:introb}: following the studies in  \cite{LW01}, but with non-trivial additional technical complications, see for instance Section \ref{chcompfun}, 
this will be done by studying the behaviour of suitable harmonic functions in the given spaces. 

\section{Preliminaries}\label{chprel}

We shall assume the reader familiar with the concepts we are going to use in this manuscript, such as Sobolev functions on metric measure spaces, $L^0$-normed modules, $\RCD$ condition, Regular Lagrangian Flows and so on. In this short preliminary section we recall some of the basic definitions, mainly just to fix the notation. Further precise references will be given throughout the text, whenever needed.

\subsection{First order Sobolev calculus}

Here we briefly recall concepts related to first order Sobolev calculus on metric measure spaces. The notion of Sobolev function we are going to use comes from \cite{Cheeger00}, see also \cite{Shanmugalingam00} and the more recent  \cite{AmbrosioGigliSavare11} where it is proposed the approach we are going to use. The concept of infinitesimal Hilbertianity come from \cite{Gigli12} and those of $L^0$-module and differential from \cite{Gigli14} (but see also \cite{GR17} for a presentation closer in spirit to the one given below).

\medskip

In this paper, a \emph{metric measure space} is a triple $(\X,\sfd,\mm)$ with $(\X,\sfd)$ complete metric space (in most cases it will be \emph{proper}, i.e.\ with bounded closed sets that are compact) and $\mm$ a reference Borel measure non-negative, non-zero and giving finite mass to bounded sets.

By $\Lip_{loc}(\X)$, $\Lip_{bs}(\X)$ we denote the space of the locally Lipschitz functions and Lipschitz functions with bounded support, respectively.

For $U\subset\X$ open and $p\geq 1$, the space $L^p_{loc}(U)$ is the collection of (equivalence classes of) functions $f:U\to\R$ such that any $x\in U$ has a neighbourhood on which $|f|^p$ is integrable (with the obvious modification for $p=\infty$).

$\pr(\X)$ is the collection of Borel probability measures on $\X$ and  $C([0,1],\X))$ the space of continuous curves on $\X$ equipped with the `sup' distance. This is complete and separable. For $t\in[0,1]$ the evaluation map $\e_t:C([0,1],\X)\to\X$ sends $\gamma$ to $\gamma_t$.

\begin{def1}[Test plan]
Let \((\X,\di,\m)\) be a metric measure space. A probability measure \(\ppi\in \mathscr{P}(C([0,1],\X))\) is said to be a test plan on \(\X\) provided the following two properties are satisfied:
\begin{itemize}
	\item there exists a constant \(C>0\) such that \((\e_t)_*\ppi\leq C\m\) for every \(t\in[0,1]\);
	\item it holds that \(\int_0^1\int|\dot{\gamma}_t|^2\,\d\ppi(\gamma)\,\d t<+\infty\).
\end{itemize}
(for the concept of \emph{metric speed} $|\dot\gamma_t|$ see \cite{Ambr90}). For $U\subset\X$ open we say that $\ppi$ is a test plan on $U$ if it is concentrated on curves taking values in $U$.
\end{def1}
Sobolev functions can be introduced by duality with test plans:
\begin{def1}[Sobolev class]
Let \((\X,\di,\m)\) be a metric measure space and $U\subset\X$ open. We say that $f:U\to\R$ Borel belongs to the local Sobolev class $S^2_{loc}(U)$ provided there is $G\in L^2_{loc}(U)$, $G\geq 0$ such that 
\[
\int|f(\gamma_1)-f(\gamma_0)|\,\d\ppi(\gamma)\leq \iint_0^1\int G(\gamma_t)|\dot{\gamma}_t|\,\d t\,\d\ppi(\gamma)
\]
holds for any test plan $\ppi$ on $U$. Any such $G$ is called \emph{weak upper gradient} of $f$.
\end{def1}
It can be proved that there is a minimal, in the $\mm$-a.e.\ sense, weak upper gradient of any $f\in S^2_{loc}(\X)$: it will be denoted $|\d f|$. We then define:
\begin{def1}[Sobolev space]
Let \((\X,\di,\m)\) be a metric measure space and $U\subset\X$ open. The space $W^{1,2}_{loc}(U)$ is defined as $L^2_{loc}(U)\cap S^2_{loc}(U)$ and $W^{1,2}(U)\subset W^{1,2}_{loc}(U)$ is the subset of those $f$ with $|f|,|\d f|\in L^2(U)$. Also, $W^{1,2}_0(U)\subset W^{1,2(U)}$ is the $W^{1,2}$-closure of the space of functions with support bounded and contained in $U$.
\end{def1}
The space $W^{1,2}(U)$ is a Banach space when equipped with the norm
\[
\|f\|_{W^{1,2}(U)}^2:=\|f\|_{L^2(U)}^2:+\||\d f|\|_{L^2(U)}^2.
\]
In general it is not Hilbert and this motivates:
\begin{def1}
We say that $(\X,\sfd,\mm)$ is infinitesimally Hilbertian if $W^{1,2}(\X)$ is an Hilbert space.
\end{def1}
To associate a differential to Sobolev functions it is convenient to introduce the following:

\begin{def1}[\(L^0\)-normed \(L^0\)-module]
An \(L^0(\m)\)-normed \(L^0(\m)\)-module is a quadruple \((\mathscr{M},\tau,\cdot,|\cdot|)\) with the following properties:
\begin{itemize}
	\item \((\mathscr{M},\tau)\) is a topological vector space;
	\item $\cdot$ is the multiplication by \(L^0\) functions, i.e.\ a map \(\cdot:L^0(\m)\times \mathscr{M}\rightarrow\mathscr{M}\) that is bilinear and
	\[f\cdot(g\cdot v)=(fg)\cdot v,\text{ and } \hat{1}\cdot v=v\]
	for every \(f,g\in L^0(\m)\) and \(v\in \mathscr{M}\), where \(\hat{1}\) is the function identically equal to 1;
	\item $|\cdot|$ is the pointwise norm, i.e.\ a map \(|\cdot|:\mathscr{M}\rightarrow L^0(\m)\) satisfying
	\begin{align*}
	|v|\geq 0 \hspace{0.3cm}&\m-\text{a.e. for every }v\in\mathscr{M},\\
	|f\cdot v|=|f||v| \hspace{0.3cm}&\m-\text{a.e. for every } f\in L^0(\m)\text{ and }v\in\mathscr{M};\\
	| v+w|\leq |v|+|w| \hspace{0.3cm}&\m-\text{a.e. for every } v,w\in\mathscr{M};
	\end{align*}
	\item the distance
\[\di_\mathscr{M}(v,w):=\int|v-w|\land f\,\d\mm, \]
where $f$ is a fixed function in $L^1(\mm)$ that is strictly positive $\mm$-a.e., is complete and induces the topology \(\tau\) (the choice of the particular $f$ is irrelevant for this purpose).
\end{itemize}
\end{def1}
Then the space $L^0(T^*\X)$ of 1-forms can be introduced via:
\begin{thm}
Let $(\X,\sfd,\mm)$ be a metric measure space. Then there is a unique couple $(L^0(T^*\X),\d)$ with $L^0(T^*\X)$ $L^0$-module in the above sense and $\d:S^2_{loc}(\X)\to L^0(T^*\X)$ such that
\begin{itemize}
\item for any $f\in S^2_{loc}(\X)$ the pointwise norm of $\d f$ is $\mm$-a.e.\ equal to the minimal weak upper gradient $|\d f|$,
\item $L^0$-linear combinations of $\d f$'s for $f\in S^2_{loc}(\X)$ are dense in $L^0(T^*\X)$.
\end{itemize}
The couple $(L^0(T^*\X),\sfd)$ is unique up to unique isomorphism.
\end{thm}
We shall refer to  $\d f$ as the \emph{differential} of the function $f$. It satisfies the following natural properties:
\[
\begin{split}
\d f&=\d g,\qquad\mm\text{-a.e.\ on $\{f=g\}$},\\
\d f&=0,\qquad\mm\text{-a.e.\ on $f^{-1}(N)$  for $N\subset\R$ Borel negligible}\\
\d(\varphi\circ f)&=\varphi'\circ f\,\d f\qquad\mm\text{-a.e.\ for $\varphi:\R\to\R$ Lipschitz},\\
\d(fg)&=f\d g+g\d f,\qquad\qquad\mm\text{-a.e.\ for $f,g\in L^\infty_{loc}\cap S^2_{loc}(\X)$}.
\end{split}
\]

It is possible to prove that $\X$ is infinitesimally Hilbertian iff the pointwise parallogram identity
\[
2(|v|^2+|w|^2)=|v-w|^2+|v+w|^2\qquad\mm-a.e.,
\]
holds for any $v,w\in L^0(T^*\X)$. In this case by polarization we can define a \emph{pointwise scalar product} as
\[
\la v,w\ra:=\tfrac12(|v+w|^2-|v|^2-|w|^2)\qquad\mm-a.e.,
\]
and this is $L^0$-bilinear, continuous and satisfies the natural Cauchy-Schwarz inequality. In this case $L^0(T^*\X)$ can be canonically identified with its dual $L^0(T\X)$ via the Riesz isomorphism (see \cite{Gigli14}): the element corresponding to $\d f$ via such isomorphism is denoted $\nabla f$ and called \emph{gradient} of $f$. 

Elements of $L^0(T\X)$ are called vector fields. Taking the adjoint of $\d$ we can define the divergence operator:
\begin{def1}[Divergence]
Let $(\X,\sfd,\mm)$ be infinitesimally Hilbertian (but this is not really necessary for this definition), $U\subset\X$ open and $v\in L^0(T\X)$. We say that $v\in D(\div_{loc},U)$ provided there is $f\in L^2_{loc}(U)$ such that
\[
\int \d g(v)\,\d\mm=-\int fg\,\d\mm\qquad \forall g\in W^{1,2}_0(U).
\] 
In this case the function $f$, that is easily seen to be unique, is denoted $\div (v)$. In the case $U=\X$ we simply write $v\in D(\div_{loc})$
\end{def1}
Taking the divergence of the gradient we get the Laplacian:
\begin{def1}[Laplacian]\label{def:lapl}
Let $(\X,\sfd,\mm)$ be infinitesimally Hilbertian (here this matters), $U\subset \X$ open and $f\in W^{1,2}_{loc}(\X)$. We say that $f\in D(\Delta_{loc},U)$ provided $\nabla f\in D(\div_{loc},U)$, i.e.\ if there is $h\in L^2_{loc}(\X)$ such that
\[
\int \la \nabla f,\nabla g\ra\,\d\mm=-\int gh\,\d\mm\qquad \forall g\in W^{1,2}_0(U).
\]
In this case the function $h$ is denoted $\Delta f$. In the case $U=\X$ we simply write $f\in D(\Delta_{loc})$.
\end{def1}

\subsection{$\RCD$ spaces and bits of second order calculus}
The following definition has been proposed in \cite{Gigli12}:
\begin{def1}[$\RCD(K,N)$ spaces]
A metric measure space $(\X,\sfd,\mm)$ is an $\RCD(K,N)$ space, $K\in\R$, $N\geq1$ provided it is a ${\sf CD}(K,N)$ space in the sense of Lott-Sturm-Villani (see \cite{Sturm06I}, \cite{Sturm06II}, \cite{Lott-Villani09}) and it is infinitesimally Hilbertian.
\end{def1}
For more details on the concept of ${\sf CD}(K,N)$ space and detailed bibliographical references we refer to the surveys \cite{Villani2017}, \cite{AmbrosioICM}, \cite{DGG}. Here we only recall some  of the properties of $\RCD$ spaces we shall most frequently use.

We start with the Bishop-Gromov inequality. For $k\in\R$ we denote by $\sink_k:\R\to\R$ the function defined by
\[
\sink_k''+k\sink=0,\qquad\sink_k(0)=0,\quad\sink_k(0)'=1.
\]
In particular we have
\[
\begin{split}
\sink_1(z)=\sin(z),\qquad\qquad\sink_0(z)=z,\qquad\qquad \sink_{-1}(z)=\sinh(z)
\end{split}
\]
for every $z\in\R$. The following result has been established in   \cite{Sturm06II}.
\begin{thm}[Bishop-Gromov Inequality]
Let \((\X,\di,\m)\) be a \({\sf CD}(K,N)\) m.m.s. with \(K\in\R\) and \(N> 1\). Fix $p\in\X$ and let $\mu:=\sfd(\cdot,p)_*\mm$. Then $\mu=s\mathcal L^1$ for some function $s:[0,{\rm diam}(\X)]\to\R^+$ such that
\begin{equation}
\label{eq:BG}
r\quad\mapsto\quad s(r) \big(\sink_{\frac{K}{N-1}}(r)\big)^{1-N}\qquad\text{ is not increasing on } [0,{\rm diam}(\X)].
\end{equation}
Also, the following integrated version of the monotonicity holds:
\begin{equation}
\label{eq:BGvol}
R\quad\mapsto\quad \tfrac{\mm(B_R(p))}{\int_0^R\big(\sink_{\frac{K}{N-1}}(r)\big)^{N-1}\,\d r}\qquad\text{ is not increasing on } [0,{\rm diam}(\X)].
\end{equation}

\end{thm}
In relating Sobolev calculus to metric properties of the space, it is useful to recall that $\RCD$ spaces have the \emph{Sobolev-to-Lipschitz} property, i.e.:
\[
\text{any $f\in W^{1,2}_{loc}(\X)$ with $|\d f|\leq 1$ has a 1-Lipschitz representative.}
\]
Analysis on $\RCD$ spaces is strongly based upon the concept of \emph{test function}, that provides a reasonable replacement for that of smooth function in our setting. Following \cite{Savare13} we define the space $\Test_{loc}(\X)$ of \emph{local test functions} on $\X$:
\[
\Test_{loc}(\X):=\big\{f\in L^\infty_{loc}(\X)\cap \Lip_{loc}(\X)\cap D(\Delta_{loc})\ :\ \Delta f\in W^{1,2}_{loc}(\X)\big\}.
\]
These functions should be thought of as `the smoothest functions' on $\X$, in a sense, and they replace what on a Riemannian manifolds are $C^\infty$ functions. Two notable properties, valid on arbitrary $\RCD(K,\infty)$ spaces are:
\begin{itemize}
\item[1)] \emph{Cut-off}. For $K\subset \X$ compact and $U\supset K$ open there is $f\in\Test_{loc}(\X)$ identically 1 on $K$ and with support in $U$. See \cite{AmbrosioMondinoSavare13-2}.
\item[2)] \emph{Algebra}. For $f,g\in\Test_{loc}(\X)$ we have $fg\in \Test_{loc}(\X)$. See \cite{Savare13}.
\end{itemize}
The latter property follows from the fact that for $f\in \Test_{loc}(\X)$ we have $|\d f|^2\in W^{1,2}_{loc}(\X)$, which in turn is useful to define second order Sobolev spaces (see \cite{Gigli14}).
\begin{def1}[The spaces $W^{2,2}_{loc}(\X)$ and $H^{2,2}_{loc}(\X)$] Let $(\X,\sfd,\mm)$ be a $\RCD(K,\infty)$ space and $f\in W^{1,2}_{loc}(\X)$. We say that $f\in W^{2,2}_{loc}(\X)$ provided there is a symmetric tensor $A\in L^2_{loc}(T^{\otimes 2}\X)$ such that
\[
\begin{split}
\int h A(\nabla g,\nabla g)\,\d\mm=\int -h\la\nabla f,\nabla g\ra\,\div(h\nabla g)-\tfrac12h\la\nabla f,\nabla(|\d g|^2)\ra\,\d\mm 
\end{split}
\]
for any $g\in \Test_{loc}(\X)$ and $h\in \Lip_{bs}(\X)$. The tensor $A$ is called Hessian of $f$ and denoted $\Hess f$.

The subspace $H^{2,2}_{loc}(\X)\subset W^{2,2}_{loc}(\X)$ is the collection of all functions $f\in W^{2,2}_{loc}(\X)$ such that: for any $U\subset\X$ open bounded there is a sequence $(f_n)\subset \Test_{loc}(\X)$ such that
\[
\|f_n-f\|_{L^2(U)}+\||\d f_n-\d f|\|_{L^2(U)}+\||\Hess f_n-\Hess f|_{\sf HS}\|_{L^2(U)}\to 0
\]
as $n\to\infty$, where here and in what follows we denote by $|\cdot|_{\sf HS}$ the pointwise norm in $L^2_{loc}(T^{\otimes 2}\X)$.
\end{def1}
Notice that implicit in the above there is the non-trivial fact that $\Test_{loc}(\X)\subset W^{2,2}_{loc}(\X)$.

\medskip

An important tool that we will use in our analysis is that of Regular Lagrangian Flow of a vector field. It provides a robust alternative to the classical Cauchy-Lipschitz theory in our setting. We shall need this concept only for a time independent vector field, thus we will recall the main definitions and results in this case.
\begin{def1}[Regular Lagrangian Flow]
Let \(v\in L^0(T\X)\). We say that a map \(\Fl:\R\times \X\rightarrow \X\) is a Regular Lagrangian Flow of  \(v\) if the following are satisfied:
\begin{itemize}
	\item For every $T>0$ there exists \(C_T>0\) such that
\begin{equation}
\label{eq:boundcompr}
\Fl_{t*}\m\leq C_T\m \text{ for every } t\in[-T,T]; 
\end{equation}
	\item for \(\m\)-a.e. \(x\in \X\) the function \(\R\ni t\rightarrow \Fl_t(x)\) is continuous and satisfies \(\Fl_0(x)=x\);
	\item for every \(f\in \Test_{loc}(\X)\) it holds that for \(\m\)-a.e. \(x\in \X\) the function \(\R\ni t\rightarrow f\circ \Fl_t(x)\) is absolutely continuous and
\begin{equation}\label{RLFder}
\frac{\d}{\d t}f\circ \Fl_t(x)=\d f(v) \circ \Fl_t(x)\qquad a.e.\ t.
\end{equation}
\end{itemize}
\end{def1}
Existence and uniqueness of Regular Lagrangian Flows can be established under suitable regularity assumption on the vector field $v$. Among other things, a control on the covariant derivative $\nabla v$ is required. We shall denote by $W^{1,2}_{C}(T\X)$ the space of $L^2$ vector fields with covariant derivative in $L^2$, referring to \cite{Gigli14} for the definition. Here we just recall that if $f\in W^{2,2}(\X)$, then $\nabla f\in W^{1,2}_{C}(T\X)$ and $\nabla(\nabla f)=\Hess f$ and that for $v\in W^{1,2}_C(T\X)$ and $\eta\in \Lip_{bs}(\X)$ we have $\eta v\in W^{1,2}_C(T\X)$ with $\nabla(\eta v)=\eta\nabla v+\nabla\eta\otimes v$.

With this said we have:
\begin{thm}
Let $(\X,\sfd,\mm)$ be an $\RCD(K,\infty)$ space and $v\in L^\infty(T\X)\cap W^{1,2}_C(T\X)\cap D(\div)$ be with $\div(v)\in L^\infty(\X)$.

Then there exists a unique Regular Lagrangian Flow $\Fl$ of $v$. Uniqueness is intended as: if both $\Fl$ and $\tilde\Fl$ are two such flows, then for $\mm$-a.e.\ $x$ we have $\Fl_t(x)=\tilde\Fl_t(x)$ for any $t\in \R$.

Finally, for $\mm$-a.e.\ $x\in\X$ the curve $t\mapsto\Fl_t(x)$ is absolutely continuous and
\begin{equation}
\label{eq:speedrlf}
\ms(\Fl_\cdot(x),t)=|v|\circ F_t(x)\qquad a.e.\ t\in[0,1],
\end{equation}
where by $\ms(\gamma,t)$ we indicate the metric speed of the curve $\gamma$ at the time $t$.
\end{thm}
Both the definition and the existence and uniqueness result for Regular Lagrangian Flows come from \cite{Ambrosio-Trevisan14}, that in turn is strongly inspired by the earlier works  \cite{DiPerna-Lions89} and  \cite{Ambrosio04} in the Euclidean setting.

\section{A general splitting principle}

\subsection{Producing a coordinate function}

In the practice of studying rigidity properties of spaces with lower Ricci bounds, one often finds out a function $u$ with somehow controlled gradient and Laplacian and for which equality holds in the Bochner inequality, and this in turn gives information about the structure of the Hessian of $u$. The typical informations of $u$ are that the identities
\[
\begin{split}
|\d u|&=\varphi\circ u,\\
\Delta u&=\xi\circ u,\\
\Hess(u)&=\zeta\circ u\,|\d u|\,\Id+\tilde\zeta\circ u\,|\d u|\,\e_1\otimes\e_1,
\end{split}
\]
where $\e_1=\tfrac{\nabla u}{|\nabla u|}$ on $|\nabla u|>0$, for suitable functions $\varphi,\zeta,\tilde\zeta,\xi:\R\to\R$.

As we shall see, when this occurs the space splits as warped product $\R\times_w\X'$ and the `$\R$-coordinate' of the isomorphism is the post-composition of $u$ with a suitable function $\eta:\R\to\R$.

To have a better understanding of the warped product we will ultimately obtain, it is convenient to identify right now who is such suitable post-composition. This is the scope of the following lemma, whose proof only relies in handling  chain rules: 
\begin{lemma}[Producing a coordinate function]\label{le:coord}
Let $(\X,\sfd,\mm)$ be a $\RCD(K,N)$ space and $u\in H^{2,2}_{loc}(\X)$ be with 
\begin{equation}
\label{eq:nablaf}
|\d u|=\varphi\circ u\qquad\mm-a.e.
\end{equation}
for some $\varphi:u(\X)\to(0,\infty)$ in $C^{1,1}_{loc}$. Put $\e_1:=\frac{\nabla u}{|\nabla u|}$ (this is $\mm$-a.e.\ well defined as $|\nabla u|>0$ $\mm$-a.e.\ as a consequence of our assumption on $\varphi$) and assume that for some locally Lipschitz functions $\zeta,\tilde\zeta:u(\X)\to\R$ we have
\begin{equation}
\label{eq:hessf}
\Hess u=\zeta\circ u |\d u|\Id+\tilde\zeta\circ u|\d u|\e_1\otimes\e_1.
\end{equation}
Then 
\begin{equation}
\label{eq:phipsi}
\varphi'=\zeta+\tilde\zeta,
\end{equation}
 any function $\eta:u(\X)\to(0,\infty)$ in $C^{1,1}_{loc}$  such that  $\eta'=\tfrac1\varphi$ is  invertible and  the function $\b:=\eta\circ u$ is in $H^{2,2}_{loc}(\X)$ with
\begin{subequations}
\label{eq:proprb}
\begin{align}
\label{eq:proprb1}
|\d \b|&=1,\\
\label{eq:proprb2}
\Hess \b&=\zeta\circ \eta^{-1}\circ \b\big(\Id-\e_1\otimes \e_1\big).
\end{align}
\end{subequations}
If moreover we have $u\in D(\Delta_{loc})$ with $\Delta u=\xi\circ u$ for some $\xi:u(\X)\to\R$ Borel locally bounded, then $\b\in D(\Delta_{loc})$ with
\begin{equation}
\label{eq:Deltab}
\Delta \b=\big(\tfrac{\xi}{\varphi}-\varphi'\big)\circ\eta^{-1}\circ \b.
\end{equation}
\end{lemma}
\begin{proof}
Since $u\in H^{2,2}_{loc}(\X)$ is with $|\d u|\in L^\infty_{loc}(\X)$ (by \eqref{eq:nablaf} )we have $|\nabla u|\in W^{1,2}_{loc}(\X)$ (recall \cite[Proposition 3.3.22]{Gigli14})), thus we can write
\[
\varphi'\circ u\,\d u\stackrel{}=\d|\d u|\stackrel{}=\Hess u(\tfrac{\d u}{|\d u|})=(\zeta\circ u+\tilde\zeta\circ u)\,\d u
\]
and since $|\d u|>0$ $\mm$-a.e., property \eqref{eq:phipsi} follows.

The fact that $\eta$ is invertible follows directly from $\eta'=\tfrac1\varphi>0$ and the fact that $\b\in H^{2,2}_{loc}(\X)$ from the chain rule noticing that:
\begin{itemize}
\item[-] if $u\in \Test_{loc}(\X)$  and $\varphi\in C^\infty_{loc}(u(\X))$, then $\varphi\circ u\in \Test_{loc}(\X)$ (by direct computation). 
\item[-] Formula \cite[Equation 3.3.32]{Gigli14}  shows that if $u$ is also locally Lipschitz, then the Hessian of $\varphi\circ u$ is locally in $L^2$.
\end{itemize}
Then \eqref{eq:proprb1} follows from $|\d \b|=\eta'\circ u|\d u|=(\eta'\varphi)\circ u$ and the choice of $\eta$. For \eqref{eq:proprb2} we use the chain rule for the Hessian \cite[Equation 3.3.32]{Gigli14} to compute
\[
\begin{split}
\Hess\b&=\eta'\circ u\Hess u+\eta''\circ u\d u\otimes\d u\\
\text{(by \eqref{eq:nablaf},\eqref{eq:hessf})}\qquad& =\zeta\circ u\,\Id+\tilde\zeta\circ u\e_1\otimes\e_1-\varphi'\circ u\e_1\otimes\e_1\\
\text{(by \eqref{eq:phipsi})}\qquad &=\zeta\circ u\,\big(\Id -\e_1\otimes\e_1\big),
\end{split}
\]
which is  \eqref{eq:proprb2}. The last claim is also a direct consequence of the assumptions and of the chain rule for the Laplacian.
\end{proof}

\subsection{Set up and statement of the splitting result}\label{se:setup}

From now on we shall assume the following:
\begin{itemize}
\item[a)] $(\X,\sfd,\mm)$ is an $\RCD(K,N)$ space, $K\in\R$, $N<\infty$, with $\supp(\mm)=\X$.
\item[b)] $\b:\X\to\R$ is a function in $H^{2,2}_{loc}\cap D(\Delta_{loc})$ such that
\begin{subequations}
\label{eq:b}
\begin{align}
\label{eq:b1}
|\d \b|&=1,\qquad\mm-a.e.\\
\label{eq:b2}
\Delta \b&=\psi_\mm\circ \b,\\
\label{eq:b3}
\Hess\b&=\psi_\sfd\circ \b\,(\Id-\e_1\otimes\e_1)\qquad\text{ where }\e_1:=\nabla\b=\tfrac{\nabla\b}{|\nabla \b|},
\end{align}
\end{subequations}
for some $\psi_\mm,\psi_\sfd:\R\to\R$ locally Lipschitz. By \eqref{eq:b1} and the Sobolev-to-Lipschitz property, $\b$ has a 1-Lipschitz representative and we shall always identify it with such representative.
\end{itemize}
These assumptions have the following rather direct consequences:
\begin{itemize}
\item[-] The Regular Lagrangian Flow $\Fl:\R\times\X\to\X$ of  $\nabla\b$ is well defined. Indeed,   for $\eta\in\Lip_{bs}(\X)$ the vector field $\eta\nabla\b$ admits a Regular Lagrangian Flow thanks to the properties \eqref{eq:b}. Then taking into account the finite speed of propagation of this flow (from \eqref{eq:b1} and \eqref{eq:speedrlf}) it is easy to see that  taking $\eta$ equal 1 on a ball of radius $R$ and then letting $R\uparrow\infty$ and using the fact that $\div(\nabla\b)=\Delta\b$ is bounded on the level sets of $\b$ we can find the desired flow $(\Fl_t)$.
\item[-] $\Fl:\R\times\X\to\X$ has a continuous representative, still denoted $\Fl$, and its Lipschitz constant on $[-T,T]\times\X$ is bounded for every $T>0$. Such Lipschitz regularity follows from the fact that the covariant derivative of $\nabla\b$ is bounded on $\b^{-1}([-T,T])$, because of \eqref{eq:b3}, and \cite[Section 2]{BS18a}.
\begin{remark}\label{re:BS}{\rm In \cite{BS18a} the authors assumed the space to be compact in order to deduce Lipschitz regularity of the flow. This was needed as they were using the main result in \cite{GT17} that, at that time, provided a necessary second-order differentiation formula on finite-dimensional and compact $\RCD$ spaces. A subsequent improvement of this paper \cite{GigTam18} established the same second-order differentiation formula in the non-compact setting, thus allowing Bru\`e-Semola's result to be extended to the non-compact setting.
}\fr\end{remark}
\item[-] The formula
\begin{equation}
\label{eq:bflusso}
\b(\Fl_t(x))=\b(x)+t\qquad\forall t\in\R,\ x\in\X
\end{equation}
holds. Indeed, we know from \eqref{RLFder} that for $\mm$-a.e.\ $x$ the function $t\mapsto\b(\Fl_t(x))$ is in $W^{1,1}_{loc}(\R)$ (in fact locally absolutely continuous, as $\b$ is Lipschitz) with derivative equal to $|\d \b|^2(\Fl_t(x))$. By integration and using \eqref{eq:boundcompr}  we see from \eqref{eq:b1} that for given $t\in\R$ the formula \eqref{eq:bflusso} holds for $\mm$-a.e.\ $x$. The claim follows by the continuity of $\Fl$ and $\b$.
\item[-] For any two $x,y\in\X$ with $\b(x)=\b(y)$ there is a Lipschitz curve joining them lying entirely on the same level set. Indeed, if $\gamma$ is any Lipschitz path connecting $x$ any $y$, say a geodesic, then $t\mapsto \Fl_{\b(x)-\b(\gamma_t)}(\gamma_t)$ is still Lipschitz (by the above Lipschitz regularity) joins $x$ and $y$ and lies on the same level set of $x,y$ (by \eqref{eq:bflusso}).
\end{itemize}
We can now define the metric measure space $(\X',\sfd',\mm')$ by putting $\X':=\b^{-1}(0)$ and then defining
\begin{equation}
\label{eq:defdp}
\sfd'(x,y)^2:=\inf\int_0^1|\dot\gamma_t|^2\,\d t,
\end{equation}
the inf being taken among all continuous curves $\gamma:[0,1]\to\X'\subset\X$ joining $x$ and $y$ and 
\begin{equation}
\label{eq:defmp}
\mm':=c^{-1}\,\Pr_*(\mm\restr{\b^{-1}([0,1])}),\qquad\text{for}\qquad c:=\int_0^1w_\mm\,\d\mathcal L^1,
\end{equation}
where $\Pr:\X\to\X'$ is defined as $\Pr(x):=\Fl_{-\b(x)}(x)$ and $w_\mm$ is defined by \eqref{eq:defwm} and \eqref{eq:normw} below. It is clear that $\sfd'$ is a (finite) distance on $\X'$ and, since clearly from the above $\Pr$ is locally Lipschitz, that it is complete and induces the same topology on $\X'$ as the one induced by $\sfd$. Then obviously $\mm'$ is a Borel measure and from \eqref{eq:T} below it easily follow that it is finite on bounded sets (and thus  Radon).

We also define
\begin{equation}
\label{eq:defT}
\begin{array}{rccc}
\sfT:&\X&\to&\R\times\X',\\
&x&\mapsto&(\b(x),\Pr(x))
\end{array}
\end{equation}
and notice that
\begin{equation}
\label{eq:T}
\text{$\sfT$ is locally Lipschitz, invertible, with locally Lipschitz inverse,}
\end{equation}
where on $\R\times\X' $ we are, for the moment, putting the  distance $\sfd_+((t,x),(s,y)):=\sfd'(x,y)+|t-s|$. Indeed, the local Lipschitz regularity of $\sfT$ is obvious. For the other inequality notice that
\begin{equation}
\label{eq:Tbilip}
\sfd(x,y)\leq \sfd(\Fl_{\b(x)}(\Pr(x)),\Fl_{\b(x)}(\Pr(y)))+ \sfd(\Fl_{\b(x)-\b(y)}(y),y),\qquad\forall x,y\in\X,
\end{equation}
and use that $t\mapsto \Fl_t(y)$ is 1-Lipschitz (by \eqref{eq:b1} and \eqref{eq:speedrlf}), that $\Fl_{z}:\X'\to\X$ is Lipschitz uniformly on $z\in[-T,T]$ for any $T>0$ (by \cite[Section 2]{BS18a} applied to the vector field $\eta\circ\b\nabla\b$ with $\eta$ smooth cut-off function with bounded support) and that $\sfd\leq\sfd'$ on $\X'$ (obviously from the definition). Observe also that
\begin{equation}
\label{eq:invT}
\text{the inverse of $\sfT$ is the map $(t,x')\mapsto\Fl_t(x')$}
\end{equation}
as it is easily seen from the definitions recalling also \eqref{eq:bflusso}.

\bigskip

With this said, we now want to equip the set $\R\times\X'$ with a warped product structure. To this aim, let us introduce the locally absolutely continuous functions $w_\mm,w_\sfd:\R\to(0,+\infty)$ defined by
\begin{subequations}
\begin{align}
\label{eq:defwm}
(\log(w_\mm))'&=\psi_\mm,\\
\label{eq:defwd}
(\log(w_\sfd))'&=\psi_\sfd,
\end{align}
\end{subequations}
normalized in such a way that
\begin{equation}
\label{eq:normw}
w_\mm(0)=w_\sfd(0)=1.
\end{equation}
Then on $\R\times\X'$ we define the (Radon) warped product measure $\mm_w$  by the formula
\begin{equation}
\label{eq:defmw}
\int f(t,x)\,\d\mm_w:=\int\Big( \int f(t,x)\,\d\mm'(x) \Big) w_\mm^2(t)\,\d\mathcal L^1(t)
\end{equation}
and the warped product distance $\sfd_w$ as
\begin{equation}
\label{eq:defdw}
\sfd_w\big((t,x),(s,y)\big)^2=\inf\int_0^1 |\dot\eta_r|^2+w_\sfd(\eta_r)|\dot\gamma_r|^2\,\d r,
\end{equation}
the inf being taken among all absolutely continuous curves $\gamma:[0,1]\to\X'$ and $\eta:[0,1]\to\R$ joining $x$ to $y$ and $t$ to $s$ respectively. Since $w_\sfd$ is continuous and strictly positive, it is clear that $\sfd_w$ is locally equivalent to (i.e.\ up to multiplicative constants controls and is controlled by) the distance $\sfd_+$ used above. In particular, it induces the product topology.

We can now state the main result of this chapter:
\begin{thm}\label{thm:genspl}
Under the assumptions $(a,b)$ stated above, the following holds.

The map $\sfT$ defined in \eqref{eq:defT} is a measure preserving isometry from $\X$ to $\R\times_w\X'$, the latter being the space $\R\times\X'$ equipped with the distance $\sfd_w$ and the measure $\mm_w$.
\end{thm}
The proof of this result will come as a result of the analysis in the following sections, the conclusion being in Section \ref{se:concl}.

\subsection{Behaviour of the measure under the flow}
Let us define the locally Lipschitz map $\R^2\ni (t,z)\mapsto\Psi_{\mm,t}(z)\in\R$ as
\begin{equation}
\label{eq:defpsim}
\Psi_{\mm,t}(z):=\tfrac{w_\mm(z-t)}{w_\mm(z)}.
\end{equation}
Then from \eqref{eq:defwm} we get
\begin{equation}
\label{eq:ptPsi}
\partial_t\Psi_{\mm,t}(z)=-\psi_\mm(z-t) \Psi_{\mm,t}(z),\qquad\text{and}\qquad \Psi_{\mm,0}\equiv1
\end{equation}
 and
 \begin{equation}
\label{eq:pdepsim1}
\partial_t\Psi_\mm+\partial_z\Psi_\mm+\psi_\mm\Psi_\mm=0.
\end{equation}
\begin{lemma}
With the same assumptions and notation of Section \ref{se:setup} and for $\Psi_\mm$ defined as in \eqref{eq:defpsim} we have
\begin{equation}
\label{eq:pf11}
(\Fl_t)_*\mm=\Psi_{\mm,t}\circ\b  \,\mm\qquad\forall t\in\R.
\end{equation}
\end{lemma}
\begin{proof}
Recalling the simple implication "$T_*\mu=\nu$ implies $T_*(\rho\mu)=\rho\circ T^{-1}\nu$",  and by the finite speed of propagation of the flow, to conclude it is sufficient to prove that for any Lipschitz probability density $\rho$ with bounded support we have  $(\Fl_t)_*(\rho\mm)=\rho_t\mm$, where $\rho_t=(\rho\circ\Fl_{-t})\,(\Psi_{\mm,t}\circ\b)$. 

By the uniqueness result for the continuity equation (that is central for the theory of Regular Lagrangian Flows, see \cite{Ambrosio-Trevisan14}), this latter claim will follow if we prove that $(\rho_t)$ solves 
\[
\partial_t\rho_t+\div(\rho_t \nabla\b)=0\qquad a.e.\ t.
\]
Notice that  $(t,x)\mapsto \rho_t(x)$ is Lipschitz, thus the computations that we are going to perform are justified. Now observe that  letting $h\to0$ in $\frac{\rho\circ\Fl_{-t}\circ\Fl_{-h}-\rho\circ\Fl_{-t}}{h}$ we see that $\partial_t(\rho\circ\Fl_{-t})=-\la\nabla(\rho\circ\Fl_{-t}),\nabla\b\ra$, thus 
\[
\partial_t\rho_t=-\la\nabla(\rho\circ\Fl_{-t}),\nabla\b\ra\,\,\Psi_{\mm,t}\circ\b+(\rho\circ\Fl_{-t})(\partial_t\Psi_{\mm,t})\circ\b.
\]
On the other hand
\[
\begin{split}
\div(\rho_t \nabla\b)&=\la\nabla\rho_t,\nabla\b\ra+\rho_t\Delta \b\\
&=\la\nabla(\rho\circ\Fl_{-t}),\nabla \b\ra \,\Psi_{\mm,t}\circ\b+(\rho\circ\Fl_{-t})(\partial_z \Psi_{\mm,t})\circ\b\,|\d\b|^2+\rho_t\Delta \b,
\end{split}
\]
thus recalling \eqref{eq:b} and adding up we conclude that
\[
\begin{split}
\partial_t\rho_t+\div(\rho_t \nabla\b)&=(\rho\circ\Fl_{-t})\big(\partial_t\Psi_{\mm,t}+\partial_z \Psi_{\mm,t}+\psi_{\mm}\Psi_{\mm,t} \big)\circ\b\stackrel{\eqref{eq:pdepsim1}}=0,
\end{split}
\]
as desired.
\end{proof}
For $B\subset\X'$ Borel we define $\hat B\subset\X$ as
\begin{equation}
\label{eq:defhatB}
\hat B:=\Fl_\cdot^{-1}(B)=\cup_t\Fl_t^{-1}(B)=\Fl_t(B).
\end{equation}
Clearly $\hat B$ is Borel. Then we have the following:
\begin{lemma}\label{le:muac}
With the same assumptions and notation of Section \ref{se:setup} (recall in particular the definition \eqref{eq:defwm} and the normalization  \eqref{eq:normw}) the following holds.

Let $B\subset\X'$ be Borel and put $\mu:=\b_*(\mm\restr{\hat B})$. Then
\begin{equation}
\label{eq:trmu1}
({\rm tr}_t)_*\mu=\Psi_{\mm,t}\,\mu,\qquad\forall t\in\R,
\end{equation}
 where  ${\rm tr}_t:\R\to\R$  is the translation map sending   $z$ to $z+t$ and moreover
\begin{equation}
\label{eq:muwm}
\mu=\mm'(B)\, w_\mm\,\mathcal L^1.
\end{equation}
\end{lemma}
\begin{proof} If $\mm'(B)=\infty$ the conclusion follows from the definition of $\mm'$ and \eqref{eq:pf11}, thus we assume $\mm'(B)<\infty$.  By \eqref{eq:pf11} and the invariance of $\hat B$ under the flow it follows immediately that   $(\Fl_t)_*(\mm\restr{\hat B})=\Psi_{\mm,t}\circ\b\,\mm\restr{\hat B}$.  Then \eqref{eq:trmu1} follows from
\[
({\rm tr}_t)_*\mu=({\rm tr}_t)_*\b_*(\mm\restr{\hat B})\stackrel{\eqref{eq:bflusso}}=\b_*(\Fl_t)_*(\mm\restr{\hat B})=\b_*(\Psi_{\mm,t}\circ\b\,\mm\restr{\hat B})\stackrel{*}=\Psi_{\mm,t}\,\b_*(\mm\restr{\hat B})=\Psi_{\mm,t}\mu,
\]
where  the starred equality is justified by $\int \varphi \,\d\b_*(\Psi_{\mm,t}\circ\b\,\mm\restr{\hat B})=\int (\varphi\Psi_{\mm,t})\circ\b\,\d \mm\restr{\hat B}=\int \varphi\Psi_{\mm,t}\,\d\b_*(\mm\restr{\hat B})$.

Averaging \eqref{eq:trmu1} in $t$ we see that the left hand side becomes absolutely continuous w.r.t.\ $\mathcal L^1$, thus showing that $\mu\ll\mathcal L^1$, say $\mu=\rho\mathcal L^1$. Then \eqref{eq:trmu1} becomes: for any $t\in\R$ we have  $\rho(z-t)=\Psi_{\mm,t}(z)\rho(z)$ for $\mathcal L^1$-a.e.\ $z$. From this identity and $\partial_t\Psi_{\mm,t}\restr{t=0}=-\psi_\mm$ (that comes from \eqref{eq:ptPsi})  it easily follows that the distributional derivative $\rho'$ of $\rho$ satisfies  $\rho'=\psi_\mm\rho $, thus ensuring that $\rho$ is a multiple of $w_\mm$, say $\rho=cw_\mm$.

To find the value of $c$ notice that    $\b_*(\mm\restr{\hat B})([0,1])=\mm(\hat B\cap \b^{-1}[0,1]))=\mm'(B)\int_0^1w_\mm\,\d\mathcal L^1$ by definition of $\mm'$, and also that  $\b_*(\mm\restr{\hat B})([0,1])=\mu([0,1])=\int_0^1\rho\,\d\mathcal L^1=c\int_0^1w_\mm\,\d\mathcal L^1$.
\end{proof}
Let now $B\subset\X'$ be Borel as before and with $\mm'(B)<\infty$. Define $\hat B$ as in \eqref{eq:defhatB} and notice that since  $\mu:=\b_*(\mm\restr{\hat B})$ is $\sigma$-finite (in fact Radon), we can disintegrate $\mm\restr{\hat B}$ along the map $\b$, as the standard proof (see e.g.\ \cite[Section 10.6]{Bogachev07}) naturally carries over. Thus we obtain a weakly measurable family $z\mapsto \nu_z$ of probability measures on $\X$ such that $\nu_z$ is concentrated on $\b^{-1}(z)$ for $\mu$-a.e.\ $z$ (equivalently by Lemma \ref{le:muac} above: for $\mathcal L^1$-a.e.\ $z$) and so that
\begin{equation}
\label{eq:dis1}
\int\varphi\,\d\mm\restr{\hat B}=\iint\varphi\,\d\nu_z\,\d \mu(z)\stackrel{\eqref{eq:muwm}}=\mm'(B) \int \Big(\int\varphi\,\d\nu_z\Big)w_\mm(z)\,\d\mathcal L^1(z)
\end{equation}
holds for every $\varphi:\X\to\R^+$ Borel. Then we have:
\begin{lemma}\label{le:dism}
With the same assumptions and notation of Section \ref{se:setup}  the following holds.

Let  $B\subset\X'$ be Borel with $\mm'(B)<\infty$ and $\hat B$ as in \eqref{eq:defhatB}. Then the disintegration $\{\nu_z\}_{z\in\R}$ of $\mm\restr{\hat B}$ w.r.t.\ $\b$ satisfies: for any $t\in\R$ it holds
\begin{equation}
\label{eq:samedis}
(\Fl_t)_*\nu_z=\nu_{z+t}\qquad\mu-a.e.\ z,
\end{equation}
where $\mu:=\b_*(\mm\restr{\hat B})$.
\end{lemma}
\begin{proof} Fix $t\in\R$. 
We shall prove that the family $\{(\Fl_t)_*\nu_{z-t}\}$ is an admissible disintegration of $\mm\restr{\hat B}$ w.r.t.\ $\b$: this is (equivalent to) the claim. For $\mu$-a.e.\ $z$ we know that   $\nu_z$ is concentrated on $\b^{-1}(z)$, thus $\nu_{z-t}$ is concentrated on  $\b^{-1}(z-t)$ and therefore, by \eqref{eq:bflusso},  $(\Fl_t)_*\nu_{z-t}$ is concentrated on $\b^{-1}(z)$. To conclude,  with  a density argument based also on \eqref{eq:T}  it is therefore sufficient to prove that for any Borel functions $g:\X'\to\R^+$ and $h:\R\to\R^+$ we have
\begin{equation}
\label{eq:disintvariante}
\int h\circ \b\,g\circ\Pr\,\d\mm\restr{\hat B}=\int\int h\circ \b\,g\circ\Pr\,\d(\Fl_t)_*\nu_{z-t}\,\d\mu(z).
\end{equation}
%
%
%
We start claiming that it holds
\begin{equation}
\label{eq:perpf}
\int  h(z)\Psi_{\mm,t}(z)\Big(\int g\circ\Pr\,\d\nu_{z-t}\Big)\,\d\mu(z)=\int  h(z)\Psi_{\mm,t}(z)\Big(\int g\circ\Pr\,\d\nu_{z}\Big)\,\d\mu(z).
\end{equation}
Since $\Pr\circ\Fl_t=\Pr$, we have on one hand
\[
\begin{split}
\int_{\hat B} h\circ \b\, g\circ\Pr\,\d(\Fl_t)_*\mm&=\int_{\hat B}(\Psi_{\mm,t}h)\circ \b\, g\circ\Pr\,\d \mm=\int \Psi_{\mm,t}(z)h(z)\Big(\int g\circ\Pr\,\d\nu_z \Big)\,\d\mu(z)
\end{split}
\]
and on the other
\[
\begin{split}
\int_{\hat B} h\circ \b\,g\circ\Pr\,\d(\Fl_t)_*\mm&=\int_{\hat B}h\circ\b\circ\Fl_t\,g\circ\Pr\,\d\mm\\
\text{(by \eqref{eq:bflusso})}\qquad&=\int h(z+t)\Big(\int g\circ\Pr\,\d\nu_z\Big)\,\d\mu(z)\\
&=\int h(z)\Big(\int g\circ\Pr\,\d\nu_{z-t}\Big)\,\d({\rm tr}_t)_*\mu(z)\\
\text{(by \eqref{eq:trmu1})}\qquad&=\int\Psi_{\mm,t}(z)h(z)\Big(\int g\circ\Pr\,\d\nu_{z-t}\Big)\,\d \mu(z),
\end{split}
\]
thus proving our claim \eqref{eq:perpf}. Now notice that replacing $h$ with $h\Psi_{\mm,t}$ (recall that $\Psi_{\mm,t}>0$ everywhere) and using again that $\Pr\circ\Fl_t=\Pr$, from \eqref{eq:perpf} we get
\[
\int  h(z)\Big(\int g\circ\Pr\,\d((\Fl_t)_*\nu_{z-t})\Big)\,\d\mu(z)=\int  h(z) \Big(\int g\circ\Pr\,\d\nu_{z}\Big)\,\d\mu(z).
\]
To conclude we use again the fact that $(\Fl_t)_*\nu_{z-t}$ and $\nu_z$ are both concentrated on $\b^{-1}(z)$ to deduce from the above that
\[
\iint h\circ\b\, g\circ\Pr\,\d((\Fl_t)_*\nu_{z-t})\,\d\mu(z)=\iint h\circ\b\, g\circ\Pr\,\d\nu_{z}\,\d\mu(z)=\int h\circ\b\, g\circ\Pr\,\d\mm\restr{\hat B},
\]
that is the desired \eqref{eq:disintvariante}.
\end{proof}
\begin{prop}\label{prop:Tmeaspres}
With the same assumptions and notation of Section \ref{se:setup}  the following holds.

The map $\sfT:\X\to\X'\times_w\R$  (recall \eqref{eq:defT}) is measure preserving. In other words and recalling \eqref{eq:invT}, for any $\varphi:\X\to\R^+$ Borel  we have
\[
\int \varphi\,\d\mm=\int\Big(\int\varphi\,\d(\Fl_t)_*\mm'\Big)w_\mm(t)\,\d t
\]
\end{prop}
\begin{proof} Letting $B\subset\X'$ be arbitrary Borel  with $\mm'(B)<\infty$ and letting $\varphi$ be 0 outside $\hat B$, we see that to conclude it suffices to prove that we can choose $\nu_z:=\mm'(B)^{-1} (\Fl_z)_*(\mm'\restr B)$ in the disintegration formula  \eqref{eq:dis1}.

To see this, observe that from Lemma \ref{le:dism} above and Fubini's theorem we see that for a.e.\ $z$ we have $(\Fl_t)_*\nu_z=\nu_{z+t}$ for a.e.\ $t$. Fix $\bar z$ for which this holds and for which $\nu_{\bar z}$ is concentrated on $\b^{-1}(\bar z) $ and  then define $\bar\nu_{z}:=(\Fl_{z-\bar z})_*\nu_{\bar z}$ for any $z$. Then $\bar \nu_z=\nu_z$ for a.e.\ $z$, and thus the $\bar\nu_z$'s are admissible in formula  \eqref{eq:dis1}. 

To conclude it is therefore enough to show that $\bar\nu_0=\mm'(B)^{-1} \mm'\restr B$. To see this, let $C\subset B$ be Borel, define $\hat C$ as in \eqref{eq:defhatB} and recall the definition \eqref{eq:defmp} of $\mm'$ to get
\[
\begin{split}
\mm'(C)&=\Big(\int_0^1w_\mm\,\d\mathcal L^1\Big)^{-1}\mm(\hat C\cap \b^{-1}([0,1]))\\
\text{(by \eqref{eq:dis1})}\qquad&= \Big(\int_0^1w_\mm\,\d\mathcal L^1\Big)^{-1}\mm'(B)\int_0^1\bar\nu_t(\hat C)w_\mm(t)\,\d\mathcal L^1(t)=\mm'(B)\bar\nu_0(C),
\end{split}
\]
having also used the definition of $\bar\nu_t$ in the last equality. The conclusion follows by the arbitrariness of $C\subset B$.
\end{proof}

\subsection{Behaviour of the distance under the flow}\label{se:flowdist}
We shall work under the same assumptions and notation as in Section \ref{se:setup}.

We start decomposing the tangent module $L^0(T\X)$ into the submodules $V^\parallel,V^\perp$ of vector fields that are pointwise parallel/orthogonal to $\nabla\b$. Thus we put
\[
\begin{split}
V^\para&:=\{v\in L^0(T\X)\ :\ v=f\nabla \b\text{ for some }f:\X\to\R\},\\
V^\perp&:=\{v\in L^0(T\X)\ :\ \la v,\nabla\b\ra\equiv0\}.
\end{split}
\]
Given an arbitrary vector field $v\in L^0(T\X)$ it components $v^\para,v^\perp$ in $V^\para,V^\perp$ respectively are defined as (recall that $|\d \b|\equiv 1$):
\[
v^\para:=\la v ,\nabla\b\ra\,\nabla\b\qquad\text{ and }\qquad v^\perp:=v-v^\para.
\]
To better understand the structure of $V^\para,V^\perp$, it is convenient to introduce the following two classes of functions:
\[
\begin{split}
\mathcal G&:=\{g\circ\Pr\ :\ g\in L^\infty\cap W^{1,2}(\X')\text{ with bounded support}\},\\
\mathcal H&:=\{h\circ\b \ :\ h\in L^\infty\cap W^{1,2}(\R)\text{ with bounded support}\}.
\end{split}
\]
Notice that both $\mathcal G$ and $\mathcal H$ are algebra of functions. We shall typically use letters $g,h$ for functions on $\X',\R$ respectively and $\hat g,\hat h$ for $g\circ\Pr,h\circ\b$ respectively. It is clear that $\nabla\hat g\in V^\perp$ and $\nabla\hat h\in V^\para$ and thus that
\begin{equation}
\label{eq:inclusione}
\begin{split}
V^\perp\ &\  \supset\ \ \text{sub-module of the tangent module generated by $\nabla\hat g$ with $\hat  g\in\mathcal G$},\\
V^\para\ &\ \supset\ \ \text{sub-module of the tangent module generated by $\nabla\hat h$ with $\hat h\in\mathcal H$}.
\end{split}
\end{equation}
We shall prove in a moment that these inclusion are actually identities. To see this, it is convenient to introduce the algebra $\mathcal A$ of functions on $\X$ as
\[
\mathcal A:=\{\text{algebra of functions on $\X$ of the form $\sum_{i=1}^n\hat g_i\hat h_i$,\ with $n\in\N$}\}.
\]
We then have the following result:
\begin{lemma}\label{le:densaA}
With the same assumptions and notation of Section \ref{se:setup} and with the definitions just given, the following holds.

The algebra $\mathcal A$ is densely contained in $W^{1,2}(\X)$. Similarly, the algebra of functions $f:\X'\times\R\to\R$ such that $f\circ\sfT\in\mathcal A$ is densely contained in $W^{1,2}(\X'\times_w\R)$.
\end{lemma}
\begin{proof}
The second claim follows directly from \cite[Section 3.3]{GH15}. Now recall that, directly from the definition of Sobolev spaces, if $T:(\X_1,\sfd_1,\mm_1)\to(\X_2,\sfd_2,\mm_2)$ is measure preserving and biLipschitz, then $L^{-1}\|f\|_{W^{1,2}(\X_2)}\leq \|f\circ T\|_{W^{1,2}(\X_1)}\leq L\|f\|_{W^{1,2}(\X_2)}$, where $L$ is the biLipschitz constant. Thus the conclusion would follow if we knew that $\sfT$ was globally biLipschitz, as the measure preserving property comes from Proposition \ref{prop:Tmeaspres}.

In general, it might be that $\sfT$ is only locally biLipschitz, but this is sufficient to conclude, as with a cut-off argument we see that functions with bounded support are dense in $W^{1,2}$ and we can approximate functions in $W^{1,2}(\X\times_w\R)$ with bounded support with functions as in the statement with uniformly bounded support.
\end{proof}
From this density result it follows that gradients of functions in $\mathcal A$ generate the whole tangent module. Since for $f=\sum_i\hat g_i\hat h_i$ we have $\nabla f=\sum_i\hat h_i\nabla \hat g_i+\hat g_i\nabla \hat h_i$, it follows that the sub-module generated by gradients of functions in $\mathcal G$ and $\mathcal H$ is the whole $L^0(T\X)$. Hence  \eqref{eq:inclusione} improves into
\begin{equation}
\label{eq:insiemiuguali}
\begin{split}
V^\perp\ &\ =\ \ \text{sub-module of the tangent module generated by $\nabla\hat g$ with $\hat  g\in\mathcal G$},\\
V^\para\ &\ =\ \ \text{sub-module of the tangent module generated by $\nabla\hat h$ with $\hat h\in\mathcal H$}.
\end{split}
\end{equation}
In particular, we obtain:
\begin{equation}
\label{eq:charpar}
v\in L^0(T\X)\quad \text{with}\quad\la v,\nabla \hat g\ra=0\quad\forall \hat g\in\mathcal G\qquad\Rightarrow\qquad v\in V^\para
\end{equation}

Let us now pick $\varphi\in C^2(\R)$ with globally bounded second derivative, put $\tilde\b:=\varphi\circ\b$ and notice that the Regular Lagrangian Flow $(\tilde\Fl_t)$ of $\nabla\tilde\b$ is globally well defined (the uniform bound on $\varphi''$ grants, among other things, that $|\nabla\tilde\b|$ grows at most linearly, that in turn shows that $\tilde\Fl_t(x)$ does not go to infinity in finite time). We claim that for some `reparametrization map' ${\rm rep}:\R^2\to\R$ locally Lipschitz we have
\begin{equation}
\label{eq:repfl}
\tilde\Fl_t(x)=\Fl_{{\rm rep}(t,\b(x))}(x)\qquad\forall t\in\R,\ x\in\X.
\end{equation}
Since $\sfT:\X\to\X'\times\R$ is locally biLipschitz, this will follow if we show that
\begin{subequations}
\begin{align}
\label{eq:flt1}
\Pr\circ\tilde\Fl_t&=\Pr,\quad\text{on }\X, \qquad\forall t\in\R,\\
\label{eq:flt2}
\b(\tilde\Fl_t(x))&=f_t(\b(x)),\qquad\forall t\in\R,\ x\in\X,
\end{align}
\end{subequations}
for some locally Lipschitz function $f:\R^2\to\R$. 

As for the flow $(\Fl_t)$ of $\nabla\b$, we have that for any $t\in\R$ the map $\tilde\Fl_t:\X\to\X$ is Lipschitz, with a control on the Lipschitz constant locally uniformly bounded in $t$. For $\hat g\in\mathcal G$ we can compute $\partial_t(g\circ\tilde\Fl_t)=\la\nabla \hat g,\nabla\tilde\b\ra\circ\tilde\Fl_t=\varphi'\circ\b\la\nabla \hat g,\nabla \b\ra\circ\tilde\Fl_t=0$, i.e.\ $\hat g\circ\tilde\Fl_t=\hat g$ $\mm$-a.e.. Then the arbitrariness of $\hat g$ and a continuity argument give \eqref{eq:flt1}.

For \eqref{eq:flt2} start noticing that with the same continuity arguments used to deduce \eqref{eq:bflusso}, we see that for any $x\in\X$ the curve $t\mapsto \b(\tilde\Fl_t(x))$ is $C^1$ and satisfies 
\begin{equation}
\label{eq:derbfl}
\partial_t\b(\tilde\Fl_t(x))=\varphi'(\b(\tilde\Fl_t(x)))\qquad\forall t\in\R. 
\end{equation}
In other words, the function $g_t(x):=\b(\tilde\Fl_t(x))$ solves the Cauchy problem $\partial_tg_t(x)=\varphi'(g_t(x))$ with $g_0(x)=\b(x)$, hence its value at time $t$ depends only on $t$ and $\b(x)$, as in \eqref{eq:flt2}. The local Lipschitz regularity of $f$ then follows by standard ODE estimates (recall that $\varphi$ is $C^2$).

Also,  $\tilde\Fl_t:\X\to\X$ has bounded compression. Since this map is invertible with inverse $\tilde\Fl_{-t}$ that is also of bounded compression, it  admits a differential $\d\tilde\Fl_t: L^0(T\X)\to L^0(T\X)$ (recall \cite[Section 2.4]{Gigli14}) characterized by
\begin{equation}
\label{eq:diffdf}
\la \d \tilde\Fl_t(v),\d f\ra\circ\tilde\Fl_t=\la v,\d(f\circ\tilde\Fl_t)\ra\qquad\forall v\in L^0(T\X),\ f\in W^{1,2}(\X)
\end{equation}
Arguing verbatim as for \cite[Proposition 3.31]{DPG16} we obtain the following result, whose proof we omit:
\begin{lemma}
With the same assumptions and notation of Section \ref{se:setup} and with $\tilde\b$, $(\tilde\Fl_t)$ as just defined,  the following holds.

Let $v\in L^2(T\X)$ and put $v_s:=\d \tilde\Fl_s(v)$. Then the map $\R\ni s\mapsto\tfrac12|v_s|^2\circ\tilde\Fl_s\in L^1(\X)$ is $C^1$ and its derivative, intended as limit of the difference quotients both strongly in $L^1$  and pointwise a.e., is given by
\begin{equation}
\label{eq:derspeed}
\partial_s(\tfrac12|v_s|^2\circ\tilde\Fl_s)=\Hess\tilde\b(v_s,v_s)\circ\tilde\Fl_s\qquad \forall s\in\R.
\end{equation}
\end{lemma}
The special structure of the Hessian of $\tilde \b$ allows for a more explicit computation of the above. Notice that the chain rule \cite[Proposition 3.3.21]{Gigli14} grants that $\tilde\b\in H^{2,2}_{loc}(\X)$ with
\begin{equation}
\label{eq:hesstildeb}
\Hess\tilde \b=(\psi_\sfd\varphi')\circ\b\,(\Id-\e_1\otimes\e_1)+\varphi''\circ\b\,\e_1\otimes\e_1
\end{equation}
Also, let us notice that
\begin{subequations}
\begin{align}
\label{eq:dflt1}
v\in V^\perp\qquad&\Rightarrow\qquad \d\tilde\Fl_t(v)\in V^\perp\qquad\forall t\in\R,\\
\label{eq:dflt2}
v\in V^\para\qquad&\Rightarrow\qquad \d\tilde\Fl_t(v)\in V^\para\qquad\forall t\in\R.
\end{align}
\end{subequations}
Indeed, for the first notice that
\[
\la  \d\tilde\Fl_t(v),\nabla\b\ra\circ\tilde\Fl_t\stackrel{\eqref{eq:diffdf}}=\la v,\nabla(b\circ\tilde\Fl_t)\ra\stackrel{\eqref{eq:flt2}}=(\partial_xf_t)\circ\b\,\la v,\nabla\b\ra=0
\]
while for the second we pick $\hat g\in\mathcal G$ and compute
\[
\begin{split}
\la\d\tilde\Fl_t(v),\nabla\hat g\ra\circ\tilde\Fl_t\stackrel{\eqref{eq:diffdf}}=\la v,\nabla(\hat g\circ\tilde\Fl_t)\ra\stackrel{\eqref{eq:flt1}}= \la v,\nabla\hat g\ra=0,
\end{split}
\]
thus the arbitrariness of $\hat g$ and \eqref{eq:charpar} give the claim.

We are now ready to state and prove the main result of this section:
\begin{prop}\label{prop:speedproj}
With the same assumptions and notation of Section \ref{se:setup} the following holds.

Let $\ppi$ be a test plan on $\X$. Then for $\ppi$-a.e.\ $\gamma$ the curve $\Pr(\gamma)$ defined as $t\mapsto \Pr(\gamma_t)$ is absolutely continuous and 
\begin{equation}
\label{eq:speedproj}
\ms(\Pr(\gamma),t)\leq \tfrac{1}{w_\sfd(\b(\gamma_t))}\ms(\gamma,t)\qquad a.e.\ t,
\end{equation}
where we are writing $\ms(\gamma,t)$ for the metric speed of the curve $\gamma$ at the time $t$. 

Moreover, equality holds in \eqref{eq:speedproj} provided $t\mapsto\b(\gamma_t)$ is constant for $\ppi$-a.e.\ $\gamma$ i.e.\ if $\ppi$ is concentrated on curves lying on level sets of $\b$.
\end{prop}
\begin{proof} We follow the line of thought used to prove \cite[Proposition 3.23]{DPG16}, this time paying attention to the action of the flow on the parallel and perpendicular directions.

Let $\varphi(z):=-\tfrac12z^2$, $v\in L^2(T\X)$ and let $v^\para,v^\perp\in L^2(T\X)$ be its components in $V^\para,V^\perp$, respectively. Put $v_s:=\d\tilde\Fl_s(v)$, $v^\para_s:=\d\tilde\Fl_s(v^\para)$ and $v^\perp_s:=\d\tilde\Fl_s(v^\perp)$. Then by \eqref{eq:derspeed} and \eqref{eq:hesstildeb} it follows that
\begin{subequations}
\begin{align}
\label{eq:vspara}
|v^\para_s|\circ\tilde\Fl_s&=|v^\para|\,e^{-s},\\
\label{eq:vsperp}
|v^\perp_s|\circ\tilde\Fl_s&=|v^\perp|\exp\Big(\int_0^s-(\b\,\psi_\sfd\circ\b)\circ\tilde\Fl_r\,\d r\Big).
\end{align}
\end{subequations}
Formula \eqref{eq:derbfl} gives $\b\circ \tilde\Fl_r=e^{-r}\b$, thus we have
\[
\begin{split}
\int_0^s-(\b\psi_\sfd\circ\b)\circ\tilde\Fl_r\,\d r=\int_0^s-\b e^{-r}\psi_\sfd(\b e^{-r})\,\d r=\int_{\b e^{-s}}^\b -\psi_\sfd(z)\,\d z=\log\big(\tfrac{w_\sfd(\b e^{-s})}{w_\sfd(\b)}\big)
\end{split}
\]
and therefore
\begin{equation}
\label{eq:controllospeed}
|v_s|\circ\tilde\Fl_s=\sqrt{|v^\para_s|^2\circ\tilde\Fl_s+|v^\perp_s|^2\circ\tilde\Fl_s}\leq |v|\big(e^{-s}+\tfrac{w_\sfd(\b e^{-s})}{w_\sfd(\b)}\big)
\end{equation}
%
%
%
Now the inequality \eqref{eq:speedproj} follows along the same lines used in \cite[Proposition 3.33]{DPG16}:
\begin{itemize}
\item[-] The test plan admits a derivative $\ppi'_t$ as an element of a suitable pullback of the tangent module satisfying $|\ppi'_t|(\gamma)=\ms(\gamma,t)$ for $\ppi$-a.e.\ $\gamma$ and a.e.\ $t$ (see \cite[Section 2.3.5]{Gigli14}).
\item[-] The classical formula $\partial_t(\tilde\Fl_s(\gamma_t))=\d\tilde\Fl_{s,\gamma_t}(\gamma_t')$ admits a natural analogue in this setting if one works with derivatives of test plans in place of derivative of single curves (see \cite[Proposition 3.28]{DPG16}).
\item[-] Coupling these two informations with \eqref{eq:controllospeed} we see that for $\ppi$-a.e.\ $\gamma$ we have 
\[
\ms(\tilde\Fl_s(\gamma),t)\leq \ms(\gamma,t)(e^{-s}+\tfrac{w_\sfd(\b e^{-s})}{w_\sfd(\b)})\qquad a.e.\ t
\]
\item[-] The conclusion \eqref{eq:speedproj} follows letting   $s\to\infty$ in the above, noticing that $\tilde\Fl_s\to\Pr$ pointwise, recalling the normalization choice  \eqref{eq:normw} and using the lower semicontinuity of the metric speed.
\end{itemize}
For the equality case we argue along the following lines:
\begin{itemize}
\item[-] If the test plan $\ppi$ is concentrated on curves lying on level sets of $\b$, then its speed is orthogonal to $\nabla\b$ (as in the proof of Proposition \ref{prop:comploc})
\item[-] Formula \eqref{eq:vsperp} and the link between derivatives of test plans and metric speed recalled before ensure that for $\ppi$-a.e.\ $\gamma$ we have
\begin{equation}
\label{eq:peruguale}
\ms(\tilde\Fl_s(\gamma),t)= \ms(\gamma,t) \tfrac{w_\sfd(\b e^{-s})}{w_\sfd(\b)}\qquad a.e.\ t
\end{equation}
\item[-] Recall that from  \cite[Section 2]{BS18a}  if $L:=\||\Hess\b|_{\sf HS}\|_{L^\infty}<\infty$, then  we have the bi-Lipschitz estimate
\begin{equation}
\label{eq:lipfl}
e^{-L|s|}\sfd(x,y)\leq\sfd(\Fl_s(x),\Fl_s(y))\leq e^{L|s|}\sfd(x,y)
\end{equation}
valid for any $x,y\in\X$ and $s\in\R$, for the flow of $\nabla\b$. In our setting we don't know if $\Hess\b$ is bounded, but with a cut-off argument based on the fact that $|\Hess\b|_{\sf HS}$ is bounded by a function of $\b$, it is not hard to see that \eqref{eq:lipfl} still holds for $s\in[-1,1]$ and $x,y\in \b^{-1}([-1,1])$ for some constant $L$.

\item[-] We can assume that $\ppi$ is concentrated on curves lying on a bounded set. Thus for $s$ sufficiently big $(\tilde\Fl_s)_*\ppi$ is concentrated on curves lying in $\b^{-1}([-1,1])$.
\item[-] We use the identity $\Pr(x)=\Fl_{-e^{-s}\b(x)}(\tilde\Fl_{s}(x))$ and the previous item to deduce that for all $s$ sufficiently big the estimate
\[
\ms(\Pr(\gamma),t)\stackrel{\eqref{eq:lipfl}}\geq e^{-Le^{-s\b(\gamma_t)}}\ms(\tilde\Fl_s(\gamma),t)\stackrel{\eqref{eq:peruguale}}= e^{-Le^{-s\b(\gamma_t)}} \ms(\gamma,t) \tfrac{w_\sfd(\b e^{-s})}{w_\sfd(\b)}\qquad a.e.\ t
\]
holds for $\ppi$-a.e.\ $\gamma$. Letting $s\uparrow\infty$ and recalling again the normalization   \eqref{eq:normw}  we get the equality in \eqref{eq:speedproj}.
\end{itemize}
\end{proof}

\begin{cor}\label{cor:reldiff}
With the same assumptions and notation of Section \ref{se:setup}   the following holds.

Let $g\in L^2_{loc}(\X')$ be Borel and put, as before, $\hat g:=g\circ\Pr$. Then $g\in W^{1,2}_{loc}(\X')$ if and only if $\hat g\in W^{1,2}_{loc}(\X)$ and in this case
\begin{equation}
\label{eq:diffpr}
|\d \hat g|=\tfrac{1}{w_\sfd\circ\b }  |\d g|\circ\Pr\qquad\mm-a.e..
\end{equation}
\end{cor}
\begin{proof}
It is clear from Proposition \ref{prop:Tmeaspres} that $g\in L^2_{loc}(\X')$ if and only if $\hat g\in L^2_{loc}(\X)$. Now assume that $g\in W^{1,2}_{loc}(\X')$ and let $\ppi$ be a test plan on $\X$ concentrated on curves lying on some bounded set. Then, since $\Pr$ is locally Lipschitz and by Proposition \ref{prop:Tmeaspres}, we have that $\Pr_*\ppi$ is  a test plan on $\X'$, where with a little abuse of notation we are denoting by $\Pr$ the map sending the curve $\gamma$ to the curve $t\mapsto\Pr(\gamma_t)$. Since $g$ is Sobolev we have
\[
\begin{split}
\int|\hat g(\gamma_1)-\hat g(\gamma_0)|\,\d\ppi(\gamma)&=\int|  g(\eta_1)-g(\eta_0)|\,\d\Pr_*\ppi(\eta)\\
&\leq \iint_0^1|\d g|(\eta_t)|\dot\eta_t|\,\d t\,\d\Pr_*\ppi(\eta)\\
\text{(writing $\eta=\Pr(\gamma)$ and using \eqref{eq:speedproj})}\qquad& \leq\iint_0^1\tfrac1{w_\sfd(\b(\gamma_t))} |\d g|(\Pr(\gamma_t))|\dot\gamma_t|\,\d t\,\d \ppi(\gamma),
\end{split}
\]
thus proving, by the arbitrariness of $\ppi$, that $\hat g\in W^{1,2}_{loc}(\X)$ and that inequality $\leq$ holds in \eqref{eq:diffpr}.

Now assume that  $\hat g\in W^{1,2}_{loc}(\X)$ and let $\ppi'$ be a test plan on $\X'$ concentrated on curves lying on some bounded set. Fix $T>0$, and consider the push forward $\ppi$ of the plan $\ppi'\times(\tfrac1{2T}\mathcal L^1\restr{[-T,T]})$ via the map $(x,t)\mapsto \Fl_t(x)$. Then from Proposition \ref{prop:Tmeaspres}, identity \eqref{eq:bflusso} and the fact that $\Fl:\X'\times\R\to\X$ is locally Lipschitz we see that $\ppi$ is a test plan on $\X$ concentrated on curves lying on level sets of $\b$, thus since $\hat g$ is Sobolev we have
\[
\begin{split}
\int|  g(\eta_1)-g(\eta_0)|\,\d\ppi'(\eta)&=\int|\hat g(\gamma_1)-\hat g(\gamma_0)|\,\d\ppi(\gamma)\\
&\leq\iint_0^1|\d\hat g|(\gamma_t)|\dot\gamma_t|\,\d t\,\d\ppi(\gamma)\\
\text{(by \eqref{eq:speedproj} and the def. of $\ppi$)}\qquad& =\iint_0^1\Big(\tfrac1{2T}\int_{-T}^T w_\sfd(s) |\d \hat g|(\Fl_s(\eta_t))\,\d s\Big)|\dot\eta_t|\,\d t\,\d \ppi'(\eta).
\end{split}
\]
Since $\ppi'$ was arbitrary, we see that $g$ is locally Soboloev with
\[
|\d g|(x')\leq \tfrac1{2T}\int_{-T}^T w_\sfd(s) |\d \hat g|(\Fl_s(x'))\,\d s\stackrel{*}\leq |\d g|(x')\qquad\mm'-a.e.\ x',
\]
where  the starred inequality comes from the already proven inequality $\leq$ in \eqref{eq:diffpr}. Thus the starred inequality must be an equality, and this forces the equality in \eqref{eq:diffpr} to hold $\mm$-a.e.\ on $\b^{-1}([-T,T])$. The conclusion follows by the arbitrariness of $T$.
\end{proof}
\subsection{Conclusion}\label{se:concl}

\begin{lemma}\label{le:SobtoLip}
With the same notation and assumptions as in Section \ref{se:setup} the following holds.

The warped product space $\R\times_w\X'$ has the Sobolev-to-Lipschitz property.
\end{lemma}
\begin{proof}
The arguments used in \cite[Theorem 3.34]{DPG16} carry over. One first proves that $\X'$ is locally doubling (because $\X$ is and $\Pr$ is locally Lipschitz) and has the measured length property, see \cite[Definition 3.17]{GH15}  (this follows from the fact that $\X$ has such property and the estimate \eqref{eq:speedproj}).

Then \cite[Theorem 3.24]{GH15}  applies. 
\end{proof}

\begin{thm}\label{thm:split}
With the same notation and assumptions as in Section \ref{se:setup} the following holds.

The map $\sfT:\X\to\R\times_w\X'$ is a measure preserving isometry.
\end{thm}
\begin{proof}
Since we already know that $\sfT$ is measure preserving (by Proposition \ref{prop:Tmeaspres}) and both $\X$ and $\R\times_w\X'$ have the Sobolev-to-Lipschitz property, according to \cite[Proposition 4.20]{Gigli13}  it is sufficient to prove that $f\in W^{1,2}(\R\times_w\X')$ iff $f\circ \sfT\in W^{1,2}(\X)$ and in this case
\begin{equation}
\label{eq:samediff}
|\d f|_{\R\times_w\X'}=|\d(f\circ\sfT)|_\X\qquad\mm-a.e..
\end{equation}
By a density argument based on Lemma \ref{le:densaA}, it suffices to prove the above for functions $f$ such that $f\circ\sfT\in\mathcal A$. 

Corollary  \ref{cor:reldiff}  ensures that \eqref{eq:samediff} holds if $f\circ\sfT\in \mathcal G$,  while   \eqref{eq:b1}  gives that \eqref{eq:samediff}  holds for $f\circ\sfT\in \mathcal H$ (see also the arguments used in \cite[Section 6.2]{Gigli13}). The conclusion now follows exactly as in \cite[Section 6.2]{Gigli13} (see also \cite[Section 3.8]{DPG16}) using the fact that both $\X$ and $\R\times_w\X'$ are infinitesimally Hilbertian ($\X'$ is so because of identity \eqref{eq:diffpr}, then the property carries to warped products by \cite[Section 3.2]{GH15}) and that functions in $\mathcal G$ and  $\mathcal H$ have orthogonal gradients. We omit the details. 
\end{proof}

\begin{prop}\label{prop:rcdxprimo}
With the same notation and assumptions as in Section \ref{se:setup} the following holds.

Assume furthermore that for some $\bar z\in\R$ we have $\psi_\sfd\leq 0$ on $(-\infty,\bar z]$ and $\psi_\sfd\geq 0$ on $[\bar z,+\infty)$ (thus in particular  $\psi_\sfd(\bar z)=0$). Then $(\X',\sfd',\mm')$ is an $\RCD(\tfrac1{w_\sfd^2(\bar z)}K,N)$ space.
\end{prop}
\begin{proof} As already noticed, by Corollary \ref{cor:reldiff} it easily follows that $\X'$ is infinitesimally Hilbertian, so we need only to prove the ${\sf CD}(K,N)$ condition.

For $z\in\R$ let $\X'_z:=\b^{-1}(z)$ and equip it with the distance $\sfd'_z$ defined as
\[
\sfd'_z(x,y)^2:=\inf\int_0^1|\dot\gamma_t|^2\,\d t,
\]
the inf being taken among all absolutely continuous curves $\gamma:[0,1]\to \X'_z\subset\X$. We also equip $\X'_z$ with the measure $\mm'_z:=(\Fl_z)_*\mm'$. Then using the equality case in Proposition \ref{prop:speedproj} and using the fact that $\X'$ has the measured length property (briefly mentioned in the proof of Lemma \ref{le:SobtoLip} above) it is not hard to see that $\Fl_z:\X'\to\X'_z$ satisfies
\[
\sfd'_z(\Fl_z(x),\Fl_z(y))=w_\sfd(z)\sfd'(x,y)\qquad\forall x,y\in\X'.
\]
Thus if we establish that $\X'_{\bar z}$ is $\RCD(K,N)$, taking into account how the ${\sf CD}$ condition scales with the distance (see \cite[Proposition 1.4]{Sturm06II}) we conclude. With this said, replacing $\b$ with $\b+\bar z$ we can assume that $\bar z=0$ and then the goal is to prove that $\X'$ is $\RCD(K,N)$.

The key geometric property that allows us to conclude, and for which we shall use the assumption on $\psi_\sfd$, is the following:
\begin{equation}
\label{eq:strisce}
\begin{split}
&\text{Let $\mu_0,\mu_1\in\pr(\X)$ be $\ll\mm$ and $\ppi$ be such that $W_2^2(\mu_0,\mu_1)=\iint_0^1|\dot\gamma_t|^2\,\d t\,\d\ppi(\gamma)$.}\\
&\text{Assume that $\supp(\mu_i)\subset \b^{-1}([-T,T])$, $i=0,1$.}\\
&\text{Then $\ppi$ is concentrated on curves lying on $ \b^{-1}([-T,T])$.}
\end{split}
\end{equation}
To see why this holds, let $\varphi:\R\to\R$ be $C^2$ with bounded second derivative, convex, identically 0 on $[-T,T]$ and strictly positive elsewhere. Then the (Regular Lagrangian) Flow $(\tilde\Fl_t)$ of $\nabla\tilde b$ with $\tilde\b=\varphi\circ\b$ is the identity on the strip $\b^{-1}([-T,T])$ and converges to the `projection on the boundary of such strip' outside of it, namely defining $\Pr_T:\X\to\X$ as:
\[
\Pr_T(x):=\left\{
\begin{array}{ll}
x,&\quad\text{if }\b(x)\in[-T,T],\\
\Fl_{T-\b(x)}(x),&\quad\text{if }\b(x)>T,\\
\Fl_{-T-\b(x)}(x),&\quad\text{if }\b(x)< -T,
\end{array}
\right.
\]
we have $\tilde\Fl_s(x)\to \Pr_T(x)$ as $s\uparrow\infty$ for any $x\in\X$. Then arguing exactly as for Proposition \ref{prop:speedproj} we obtain that: for every test plan $\ppi$ we have that for $\ppi$-a.e.\ $\gamma$ it holds
\begin{equation}
\label{eq:prT}
\ms(\Pr_T(\gamma),t)\leq \tfrac{w_\sfd(\b(\Pr_T(\gamma_t)))}{w_\sfd(\b(\gamma_t))}\ms(\gamma,t)\qquad a.e.\ t.
\end{equation}
Now recall that by \eqref{eq:defwd} we have $\log(\tfrac{w_\sfd(z_2)}{w_\sfd(z_1)})=\int_{z_1}^{z_2}\psi_\sfd\,\d\mathcal L^1$ for every $z_1,z_2\in\R$, $z_1<z_2$ and use the assumption on $\psi_\sfd$ (with $\bar z=0$) to conclude from \eqref{eq:prT} that  for $\ppi$-a.e.\ $\gamma$ it holds
\begin{equation}
\label{eq:prT2}
\ms(\Pr_T(\gamma),t)\leq  \ms(\gamma,t)\qquad a.e.\ t.
\end{equation}
It directly follows from this that the total kinetic energy ${\sf KE}(\ppi):=\tfrac12\iint_0^1|\dot\gamma_t|^2\,\d t\,\d\ppi(\gamma)$ does not increase under the action of $\Pr_T$. In particular, this occurs for $\ppi$ as in \eqref{eq:strisce}, thus  obtaining ${\sf KE}((\Pr_T)_*\ppi)\leq \tfrac12W_2^2(\mu_0,\mu_1)$. On the other hand, we know from \cite{GigliRajalaSturm13} that for $\mu_0,\mu_1\ll\mm$ there is exactly one plan $\ppi$ for which ${\sf KE}(\ppi)\leq  \tfrac12W_2^2(\mu_0,\mu_1)$, thus we conclude that $(\Pr_T)_*\ppi=\ppi$, which is the claim \eqref{eq:strisce}.
 
A direct consequence of \eqref{eq:strisce} and the very definition of ${\sf CD}(K,N)$ spaces (see \cite{Sturm06II}) is that the space $(\X_T,\sfd_T,\mm_T)$ given by $\X_T:=\b^{-1}([-T,T])$ with $\sfd_T$ being the restriction of the distance and $\mm_T:=\frac{1}{\int_{-T}^Tw_\mm }\mm\restr{\b^{-1}(-T,T)}$ is ${\sf CD}(K,N)$ as a consequence of $(\X,\sfd,\mm)$ being so (recall that a scaling of the measure does not affect Curvature-Dimension bounds). Therefore by the stability of the ${\sf CD}$ condition (see  \cite{Sturm06II}, \cite{Villani09}  or \cite{GMS15}) to conclude it suffices to prove that for any fixed $p\in\X'$ the spaces $(\X_T,\sfd_T,\mm_T,p)$ converge to $(\X',\sfd',\mm',p)$ as $T\downarrow0$ in the pointed-measured-Gromov-Hausdorff sense.

To see this, consider the map $\Pr:\X_T\to\X'$ and notice that $\Pr(p)=p$, that $\Pr_*\mm_T=\mm'$ (by \eqref{eq:defmp}) and that 
\[
\begin{split}
|\sfd(\Pr(x),\Pr(y))-\sfd(x,y)|\leq \sfd(\Pr(x),x)+\sfd(\Pr(y),y)\leq 2T\qquad\forall x,y\in \X_T.
\end{split}
\]
This suffices (see e.g.\ \cite[Section 3.5]{GMS15}) to prove that $(\X_T,\sfd_T,\mm_T,p)$  pmGH-converge to $(\X',\sfd,\mm',p)$, which therefore is a ${\sf CD}(K,N)$ space, so we are left to prove that $\sfd=\sfd'$ on $\X'$. To see this, notice that since $(\X',\sfd,\mm')$ is ${\sf CD}(K,N)$ and $\supp(\mm')=\X'$ (as direct consequence of the assumption $\supp(\mm)=\X$), we have that $(\X',\sfd)$ is a geodesic space, i.e.\ given $x,y\in\X'$ there is a curve $\gamma:[0,1]\to\X'$ with $\int_0^1|\dot\gamma_t|^2\,\d t=\sfd^2(x,y)$. It follows by the very definition of $\sfd'$ that $\sfd'\leq \sfd$, and since the other inequality is trivially true, the proof is complete.
\end{proof}

\section{$\RCD$ spaces with positive spectrum}

\subsection{Reminders: measure-valued Laplacian and Bochner inequality}

Here we briefly recall some additional material that will be useful in obtaining the main application we have of the general splitting theorem just proved.

The first is a generalization of the notion of Laplacian that we have seen in Definition \ref{def:lapl}:
\begin{def1}
Let $(\X,\sfd,\mm)$ be an infinitesimally Hilbertian space and $f\in W^{1,2}_{loc}(\X)$. We shall say that $f\in D(\bd_{loc})$ provided there is a Radon functional $\mu$ such that
\[
\int \la \nabla f,\nabla g\ra\,\d\mm=-\int gh\,\d\mm\qquad \forall g\in \Lip(\X)\text{ with bounded support}.
\]
In this case the measure $\mu$ is denoted  $\bd f$.
\end{def1}
Recall that a \emph{Radon functional} is a linear functional $L$ from the space $C_{bs}(\X)$ of continuous functions on $\X$ with bounded support to $\R$ such that for every $K\subset\X$ there is $C_K\geq 0$ such that
\[
|L(f)|\leq C_K\sup|f|\qquad\forall f\in C_{bs}(\X)\text{ with support in }K.
\]
Thus Radon functionals should be thought of (and we shall do so) as signed measures that have finite total variation on compact sets, but that in principle might have both positive and negative parts of infinite mass, see also the discussion in \cite{CavMon20}. Notice also that in the field of non-smooth analysis some authors use the term Radon measure for what we are calling here Radon functionals, see for instance \cite[Remark 2.12]{AB18}.

There are natural compatibility conditions between these  notion of Laplacian and that in Definition \ref{def:lapl}, in particular
\begin{equation}
\label{eq:complap}
\text{if $f\in D(\bd_{loc})$ and $\bd f=g\mm$ for $g\in L^2_{loc}$, then $f\in D(\Delta_{loc})$ with $\Delta f=g$.}
\end{equation}
Both the divergence and the Laplacian obey natural calculus rules, see \cite{Gigli14}. It can also be proved that these concepts of Laplacian are linked to energy minimizers, meaning that for $U\subset\X$ open and $f\in L^2_{loc}(U)$ the following are equivalent:
\begin{itemize}
\item[i)] $f\in D(\Delta_{loc},U)$ and $\Delta f=0$,
\item[ii)] $f\in D(\bd,U)$ and $\bd f=0$,
\item[iii)] for any $g\in W^{1,2}_0$ we have 
\[
\int_U|\d f|^2\,\d\mm\leq \int_U|\d (f+g)|^2\,\d\mm.
\]
\end{itemize}
For the proof of the equivalences see \cite{Gigli12} and \cite{Gigli-Mondino12}.

\begin{def1}[Harmonic functions]
Let $(\X,\sfd,\mm)$ be infinitesimally Hilbertian, $U\subset\X$ be open and $f\in L^2_{loc}(U)$. We say that $f$ is harmonic if it satisfies any of the three equivalent conditions above.
\end{def1}
We shall use the following regularity results for harmonic functions, that extend to the $\RCD$ setting classical estimates by Cheng-Yau:
\begin{thm}[Cheng-Yau type gradient estimate - see  \cite{Hua-Kell-Xia13}] Let \((\X,\di,\m)\) be an \(\RCD(K,N)\) metric measure space with \(K\leq 0\) and \(N\in[1,\infty)\). Then there exists a constant \(C=C(N)\) such that every positive harmonic function \(u\) on geodesic ball \(B_{2R}\subset \X\) satisfies
\begin{equation}\label{chengyau}
\frac{|\d u|}{u}\leq C\frac{1+\sqrt{-K}R}{R}\hspace{0.4cm}\text{ in }B_R.
\end{equation}
\end{thm}

\medskip

We conclude this introductory section recalling (a suitable version of) the Bochner inequality.  In order to state it, we need to recall the concept of \emph{essential dimension} ${\rm dim}(\X)$ of a finite dimensional $\RCD$ space:
\begin{thm}\label{thm:dimX}
Let \(\X\) be an \(\RCD(K,N)\) space with \(K\in\R\) and \(N\in [1,+\infty)\). Then there exists an integer \(\dim (\X)\in[1, N]\) such that the tangent module \(L^0(T\X)\) has constant dimension equal to \(\dim (\X)\).
\end{thm}
The proof of this result is highly non-trivial, and ultimately coming from \cite{BS18a} (but see also \cite{Gigli13}, \cite{Mondino-Naber14}, \cite{DPR}, \cite{DPMR16}, \cite{MK16}, \cite{GP16-2}, \cite{GP16}).

We can now state:
\begin{thm}[Improved Bochner Inequality] Let \(\X\) be an \(\RCD(K,N)\) space with \(K\in \R\) and \(N\in[1,+\infty)\). Then for any \(f\in\Test_{loc}(\X)\) we have   \(|\d f|^2\in D(\bd_{loc})\) and
\begin{equation}\label{bochner}
\bd \left(\frac{|\d f|^2}{2}\right)\geq \left(|\Hess(f)|_{\sf HS}^2+K|\d f|^2+\langle \d f,\d \Delta f\rangle+\frac{\left(\Delta f-\tr\Hess(f)\right)^2}{N-\dim(\X)} \right)\m,
\end{equation}
where \(\frac{\left(\Delta f-\tr\Hess(f)\right)^2}{N-\dim(\X)}\) is taken to be \(0\) in the case \(\dim(\X)=N\).
\end{thm}
This result was proved in  \cite{Han14} (strongly based on the earlier  \cite{Gigli-Kuwada-Ohta10},\cite{AmbrosioGigliSavare11-2}, \cite{Erbar-Kuwada-Sturm13},  \cite{Gigli14}).

\medskip


\subsection{Volume of the ends and harmonic functions}\label{chvolends}

The concept of `end' borns in the smooth category of Riemannian manifolds, but it can easily be adapted to metric spaces:
\begin{def1}
Let \((\X,\di)\) be a metric space and $K\subset\X$ compact. A set $E\subset\X$ is called end of $\X$ with respect to $K$ provided:
\begin{itemize}
\item[-] $E$ is an unbounded connected component of $\X\setminus K$,
\item[-] For any  $K'\supset K$ compact the set $E\setminus K'$ has only one unbounded connected component.
\end{itemize}
Suppose that $(\X,\sfd)$ is equipped with a Radon measure $\mm\geq 0$. Then we say that an end $E$ has infinite volume if $\mm(E)=+\infty$.
\end{def1}

Notice that if $E$ is an end of $\X$ w.r.t.\ $K$ and $K'\supset K$ is compact, then the only unbounded connected component $E'$ of $E\setminus K'$ is an end w.r.t.\ $K'$. Also, in this case  $E$ has infinite volume if and only if $E'$ does.

Let \(E\) be an end of \(\X\) with respect to \(K\) and let \(p\in K\). We indicate with \(E(R):=E\cap B_R(p)\) for every \(R>\dist(E,p)\). Moreover we define \(\partial E:=\partial K\cap \bar{E}\) and \(\partial E(R):=\partial B_R(p)\cap E\).

We indicate with \(V_E(R)\) the volume of \(E(R)\) and with \(V_E(\infty)\) the volume of the end \(E\).

We conclude defining the first eigenvalue of the Laplacian:
\begin{def1}[First eigenvalue of the Laplacian]
Let \((\X,\di,\m)\) be a metric measure space. We define
\[\lambda_1:=\inf\left\{ \frac{\int_\X |\d f|^2 \,\d\m}{\int_\X|f|^2\,\d\m}:f\in W^{1,2}(\X),\ \int f^2\,\d\m\neq 0 \right\}.
\]
\end{def1}
Notice that the definition makes sense on arbitrary metric measure spaces, regardless of the linearity of the Laplacian (but we shall only work on infinitesimally Hilbertian spaces).

In this section, following the steps used by Li and Wang, we prove that an end has infinite volume if and only if there exists a non-constant bounded harmonic function on it. In order to do this we begin studying some decay estimates for a class of harmonic functions.

In \cite[Section B.1]{GV21} the following result has been established, the point being the continuity at $\partial K^r$ (see also  \cite{Bjorn-Bjorn11}). Below for $K\subset X$ and $r>0$ we put
\[
K^r:=\{x:\di(x,K)<r\}.
\]
We  consider such enlargements to gain sufficient regularity of the boundary to be sure that the harmonic function given by the statement below is continuous up to the boundary.
\begin{thm}\label{thm:exharm} Let \((\X,\di,\m)\)  be an $\RCD(K,N)$ space, $N<\infty$, $K\subset X$ a bounded subset and $r>0$. Let also $B\subset X$ be a ball containing $K^r$. Then there is $f\in W^{1,2}_0(B)\cap C(B)$ with  \(0\leq f\leq 1\) on \(B\)  that is harmonic on $B\setminus \overline{K^r}$ and equal to 1 in $\overline{K^r}$.
\end{thm}
In particular, let $E$ be an end of $X$, say w.r.t.\ a compact set $K$. Then 
\begin{equation}
\label{eq:kp}
\text{$E$ is also an end w.r.t.\ the compact set }\ K':=\overline{K^1}.
\end{equation}
 Then for $p\in K'$ and $R>0$ large enough we can apply the above with $B=B_R(p)$ with $R>0$ big enough to find $f_R \in W^{1,2}_0(B)\cap C(B)$  with  \(0\leq f_R\leq 1\) on \(B\)  that is harmonic on $B\setminus K'$ and equal to 1 in $K'$.

Thanks to the maximum principle (see \cite{Bjorn-Bjorn11} for the proof of the maximum principle in the nonsmooth setting, see also \cite{GR19} for a different proof) the functions \(f_R\) generated above are pointwise increasing in \(R\), thus we can define:
\begin{def1}\label{def:fE} The function $f_E:E\setminus K'\to[0,1]$ is defined as the pointwise limit of the $f_R$'s described above.
\end{def1}
It is easy verify that $f_E$ depends only on $E$ and $K'$, and not on the particular $p\in K'$ chosen, and that it is harmonic. The dependence of this function on $K'$ is omitted from the notation for brevity. Also, in what follows we relabel $E$ to be $E\setminus K'$. 

Our first goal is  to prove  some key decay estimates for the function \(f_E\).

\begin{lemma}\label{lem1} 
Let \((\X,\di,\m)\) be an \(\RCD(K,N)\) m.m.s.\ and \(E\)  an end of it. Assume that \(\lambda_1>0\). Then for the harmonic function  \(f_E\) given by Definition \ref{def:fE} there exists a constant \(C=C(E,K,\lambda_1)\)  such that for every \(R\) large enough the following estimates hold:
\begin{align}
\int_{E(R+1)\setminus E(R)}f_E^2\,\d\m &\leq Ce^{ -2\sqrt{\lambda_1}R},\label{eq11}\\
\int_{E(R+1)\setminus E(R)}|\nabla f_E|^2\,\d\m &\leq Ce^{ -2\sqrt{\lambda_1}R},\label{eq12}\\
\int_{E(R)}e^{2\sqrt{\lambda_1}\di (\cdot,p)}\left|\nabla f_E\right|^2\,\d\m &\leq CR\label{eq13}.
\end{align}
\end{lemma}
\begin{proof} Let $K\subset X$ be compact, put $K':=\overline{K^1}$ and, as before,  say that  $E$ is an end w.r.t.\ $K'$. Also, let $R_0>0$ and $p\in X$ fixed so that $K'\subset B_{R_0}(p)$, recall that we put $E(R):=E\cap B_R(p)$ and that the function $f_R\in W^{1,2}_0\cap C(B_R(p))$ has been defined after Theorem \ref{thm:exharm}.

\textit{Step 0.} We claim that for every \(g\in \Lip_{loc}(\X)\) which is identically 0 on $K'$ and for every \(R\) large enough it holds
\begin{equation}\label{eqharout}
\int_{E(R)}|\nabla(gf_R)|^2\,\d\m=\int_{E(R)}|\nabla g|^2f_R^2\,\d\m.
\end{equation}
Using the Leibniz and chain rules we compute
\begin{align*}
\int_{E(R)}\left|\nabla (gf_R)\right|^2\, \d\m &= \int_{E(R)}\left|\nabla g\right|^2f_R^2\,\d\m+\int_{E(R)}g^2 \left|\nabla f_R\right|^2\,\d\m+\int_{E(R)}f_R \left\langle\nabla (g^2),\nabla f_R\right\rangle \,\d\m.
\end{align*}
Now we notice that the integration by parts in
\[
\int_{E(R)}f_R \left\langle\nabla (g^2),\nabla f_R\right\rangle \,\d\m=-\int_{E(R)} g^2\divv (f_R \nabla f_R) \,\d\m=-\int_{E(R)} g^2| \nabla f_R|^2 \,\d\m
\]
is justified by the fact that $\partial(E(R))\subset (\partial K')\cup (\partial B_R(p))$ and the assumption $g\equiv0$ on $K'$ and $f_R\in W^{1,2}_0(B_R(p))$. The claim \eqref{eqharout} follows.

The same line of thought proves that
\begin{equation}\label{eqharout2}
\int_{E}|\nabla(gf_E)|^2\,\d\m=\int_{E}|\nabla g|^2f_E^2\,\d\m
\end{equation}
if $g$ is as above and moreover in $W^{1,2}(\X)$.

\textit{Step 1.} Let \(\xi:=\exp(\sqrt{\lambda_1}\di(\cdot,p))\). We claim that for some $C=C(E,K,\lambda_1)$ we have
\begin{equation}\label{eq14}
\int_E \xi^{2\delta} f_E^2\,\d\m\leq \frac{C}{(1-\delta)^2}\qquad\forall \delta\in(0,1).
\end{equation}
Let $\varphi\in \Lip\cap C_{b}(\X)$ be defined as $\varphi:=\hat\varphi\circ\di(\cdot,p)$, where
\[\hat\varphi(z):=\begin{cases}0 & \mbox{if }z\leq R_0, \\\frac{z-R_0}{R_0}  & \mbox{if } z\in[R_0,2R_0], \\ 1 & \mbox{if }z\geq 2R_0,\end{cases}\]
and let \(R>2R_0\). By \((\ref{eqharout})\) and the Cauchy-Schwarz inequality we have that
\begin{align*}
\int_{E(R)}\left|\nabla (\varphi\xi^\delta f_R )\right|^2 \,\d\m &=\int_{E(R)} |\nabla (\varphi\xi^\delta ) |^2f_R^2\, \,\d\m\\
&\leq (1+\varepsilon)\int_{E(R)}\varphi^2 |\nabla \xi^\delta |^2f_R^2\,\d\m+\left( 1+\tfrac{1}{\varepsilon}\right)\int_{E(R)}\xi^{2\delta}\left| \nabla\varphi\right|^2f_R^2\,\d\m\\
&\leq (1+\varepsilon)\delta^2\lambda_1\int_{E(R)} \varphi^2\xi^{2\delta}f_R^2\,\d\m+\left(1+\tfrac{1}{\varepsilon}\right)\tfrac{1}{R_0^2}\int_{E(2R_0)\setminus E(R_0)}\xi^{2\delta}f_R^2\,\d\m.
\end{align*}

By definition of \(\lambda_1\) we have $\lambda_1\int_{E(R)} \varphi^2\xi^{2\delta}f_R^2\,\d\m\leq \int_{E(R)} |\nabla (\varphi\xi^\delta f_R ) |^2\,\d\m$ and thus
\begin{align*}
&\lambda_1\left(1-(1+\varepsilon)\delta^2\right)\int_{E(R)} \varphi^2\xi^{2\delta}f_R^2\,\d\m \leq (1+\tfrac{1}{\varepsilon} )\tfrac{1}{R_0^2}\int_{E(2R_0)\setminus E(R_0)}\xi^{2\delta}f_R^2\,\d\m,
\end{align*}
and choosing \(\varepsilon=\frac{1-\delta}{\delta}\), recalling that $0\leq f_R\leq 1$ we obtain
\begin{equation}
\label{eqpt1}
\lambda_1(1-\delta)^2\int_{E(R)} \varphi^2\xi^{2\delta}f_R^2\,\d\m \leq \frac{1}{R_0^2}\int_{E(2R_0)\setminus E(R_0)}\xi^{2\delta}f_R^2\,\d\m\leq C.
\end{equation}

Since \(\varphi\) is positive and identically equal to \(1\) on \(E(R)\setminus E(2R_0)\), we have that
\[
\begin{split}
\int_{E(R)}\xi^{2\delta}f_R^2\,\d\m&=\int_{E(R)\setminus E(2R_0)}\xi^{2\delta}f_R^2\,\d\m+\int_{E(2R_0)}\xi^{2\delta}f_R^2\,\d\m\\
&\leq \int_{E(R)} \varphi^2\xi^{2\delta}f_R^2\,\d\m+C\stackrel{\eqref{eqpt1}}\leq \frac C{(1-\delta)^2}+C= \frac C{(1-\delta)^2}
\end{split}
\]
and letting $R\to\infty$ we get the claim \eqref{eq14}.

\textit{Step 2.} Put for simplicity 
\begin{equation}
\label{eq:defF}
F(R):=\int_{E(R)}\xi^2f_E^2\,\d\m\qquad \forall R>0.
\end{equation}
Let  \(R_0<R_1<R\). We claim that  for any \(t\in(0,R-R_1)\) it holds 
\begin{equation}
\label{eq15}
\tfrac{2\sqrt{\lambda_1}t}{(R-R_1)^2}(F(R-t)-F(R_1))\leq \tfrac{2\sqrt{\lambda_1}(R_1-R_0)+1}{(R_1-R_0)^2}(F(R_1)-F(R_0)) +\tfrac{1}{(R-R_1)^2}(F(R)-F(R_1))
\end{equation}
To see this let $\psi\in \Lip(\X)$ be defined as $\psi:=\hat\psi\circ \di_p$ (here and below $\di_p=\di(\cdot,p)$), where
\[
\hat\psi(z):=\begin{cases}\frac{z-R_0}{R_1-R_0}  & \mbox{if }z\in [R_0,R_1], \\ \frac{R-z}{R-R_1} & \mbox{if }z\in [R_1,R], \\ 0 & \mbox{otherwise}.\end{cases}\]
By definition of \(\lambda_1\), \eqref{eqharout2} and the identity $\nabla \xi=\sqrt{\lambda_1}\xi\nabla \di_p$ we get
\begin{align*}
\lambda_1\int_E \psi^2\xi^2f_E^2\,\d\m&\leq \int_E\left| \nabla\left( \psi\xi f_E \right) \right|^2\,\d\m=\int_E\left| \nabla\left( \psi\xi\right) \right|^2f_E^2\,\d\m\\
&= \int_E\left| \nabla \psi\right|^2 \xi^2f_E^2\,\d\m+\lambda_1 \int_E \psi^2 \xi^2 f_E^2\,\d\m+ 2\sqrt{\lambda_1}\int_E\psi\xi^2\langle \nabla\psi,\nabla \di_p\rangle f_E^2\,\d\m,
\end{align*}
that can be rewritten as
\[-2\sqrt{\lambda_1}\int_E\psi\xi^2\langle \nabla\psi,\nabla \di_p\rangle f_E^2\,\d\m\leq \int_E\left| \nabla \psi\right|^2 \xi^2f_E^2\,\d\m.\]
Using the explicit expression of $\psi$, this can be further rewritten as
\[
\begin{split}
&\frac{2\sqrt{\lambda_1}}{(R-R_1)^2}\int_{E(R)\setminus E(R_1)}(R-\di_p)\xi^2f_E^2\,\d\m\\
&\leq \frac{2\sqrt{\lambda_1}}{(R_1-R_0)^2} \int_{E(R_1)\setminus E(R_0)}\underbrace{(\di_p-R_0)}_{(\leq R_1-R_0\ \text{on}\ E(R_1)\setminus E(R_0))}\xi^2f_E^2\,\d\m\\
&\qquad\qquad+\frac{1}{(R_1-R_0)^2}\int_{E(R_1)\setminus E(R_0)}\xi^2f_E^2\,\d\m+\frac{1}{(R-R_1)^2}\int_{E(R)\setminus E(R_1)}\xi^2f_E^2\,\d\m\\
&\leq \text{Right Hand Side of \eqref{eq15}}.
\end{split}
\]
Then to get \eqref{eq15} notice that for  \(t\in (0,R-R_1)\) we have  \(t\leq R-\di_p\) on \(E(R-t)\setminus E(R_1)\).

\textit{Step 3.} We claim that there exists a constant \(C=C(E,K,\lambda_1)>0\) such that 
\begin{equation}\label{eq17}
F(R) \leq CR\qquad\text{ for all $R$ large enough and $F$ as in \eqref{eq:defF}}.
\end{equation}
To see this pick    $t=1$ and replace $R$ with $R+1$ in \eqref{eq15} to get, after easy manipulation, that 
\[
F(R)-F(R_0+1)\leq C R^2+\frac{1}{2\sqrt{\lambda_1}}(F(R+1)-F(R_0+1))
\]
(observe that $R_0$ depends only on $E$ and $K$).  By iteration we get
\begin{equation}
\label{eq:conk}
\begin{split}
F(R)-F(R_0+1)&\leq C\sum_{i=1}^k\frac{(R+i)^2}{2^{i-1}}+(2\sqrt{\lambda_1})^{-k}\big(F(R+k)-F(R_0+1)\big)\\
&\leq CR^2+(2\sqrt{\lambda_1})^{-k}\big(F(R+k)-F(R_0+1)\big)\qquad\forall k\in\N,\ k>0.
\end{split}
\end{equation}
Now notice that for any $\delta\in(0,1)$ we have
\[
\begin{split}
F(R+k)-F(R_0+1)\leq e^{2\sqrt{\lambda_1}(R+k)(1-\delta)}\int_{E(R+k)\setminus E(R_0+1)} \xi^{2\delta} f_E^2\,\d \m\stackrel{\eqref{eq14}}\leq \frac {Ce^{2\sqrt{\lambda_1}(R+k)(1-\delta)}}{(1-\delta)^2},
\end{split}
\]
thus picking $\delta$ so that $2\sqrt{\lambda_1}(1-\delta)<\log(2\sqrt{\lambda_1})$, by letting $k\to\infty$ in \eqref{eq:conk} we conclude that 
%
%
%
%
\begin{equation}
\label{eq:subottimo}
F(R)\leq CR^2.
\end{equation}
To improve this, pick $t=\frac R2$ and $R_1=R_0+1$ in \eqref{eq15} to obtain, after little manipulation, that
\[
\tfrac1R\big(F(\tfrac R2)-F(R_0+1)\big)\leq C\big(1+\tfrac1{R^2}(F(R)-F(R_0+1))\big)\stackrel{\eqref{eq:subottimo}}\leq C\qquad \forall R\text{ large enough.}
\]
This  is (equivalent to) our claim \eqref{eq17}.

\textit{Step 4.} We claim that   \((\ref{eq11})\) holds. To see this,  in \((\ref{eq15})\) pick   \(t=\frac{2}{\sqrt{\lambda_1}}\), $R+t$ in place of $R$ and \(R-t\) in place of $R_1$: with little manipulation  we deduce that
\[
\frac1{t^2}\big(F(R)-F(R-t)\big)\leq \underbrace{\frac CR\big(F(R-t)-F(R_0)\big)}_{\leq C\ \text{by \eqref{eq17}}}+\frac1{4t^2}\underbrace{\big(F(R+t)-F(R-t)\big)}_{=(F(R+t)-F(R))+(F(R)-F(R-t))}
\]
and thus that $F(R)-F(R-t)\leq C+\frac13(F(R+t)-F(R))$. Iterating we get
\[
\begin{split}
F(R)-F(R-t)\leq C\sum_{i=0}^{k-1}\frac1{3^i}+\frac1{3^k}(F(R+kt)-F(R+(k-1)t))\qquad \forall k\in \N, k>0.
\end{split}
\]
Recalling \eqref{eq17} and letting $k\to\infty$ we conclude that $F(R)-F(R-t)\leq C$ for all $R$ large enough (recall that here $t=\frac{2}{\sqrt{\lambda_1}}$). The claim \eqref{eq11} easily follows.

\textit{Step 5.} We claim that  \((\ref{eq12})\) and \((\ref{eq13})\) hold.

Let  $R>R_0+1$ and \(\zeta\in \Lip_{bs}(\X)\) be defined as $\zeta:=\hat\zeta\circ\di(\cdot,p)$, where
\[\hat \zeta(z):=\begin{cases} z-(R-1)  & \mbox{if }z\in[R-1,R] \\ 1 & \mbox{if }z\in[R,R+1] \\ R+2-z & \mbox{if }z\in[R+1,R+2] \\ 0 & \mbox{otherwise }
\end{cases}\]
Then using that  \(f_E\) is harmonic we get the standard estimate
\begin{align*}
\int_E\zeta^2|\nabla f_E|^2\,\d\m &=-2\int_E\zeta f_E\langle \nabla\zeta,\nabla f_E\rangle \,\d\m\leq \frac{1}{2}\int_E\zeta^2|\nabla f_E|^2\,\d\m+2\int_E |\nabla \zeta|^2f_E^2\,\d\m
\end{align*}
that gives $\int_E\zeta^2|\nabla f_E|^2\,\d\m \leq 4\int_E |\nabla \zeta|^2f_E^2\,\d\m\leq 4\int_{E(R+2)\setminus E(R-1)}f_E^2\,\d\m$. Thus
\[
\int_{E(R+1)\setminus E(R)}|\nabla f_E|^2\,\d\m\leq \int_E\zeta^2|\nabla f_E|^2,\d\m \leq  4\int_{E(R+2)\setminus E(R-1)}f_E^2\d\m\stackrel{\eqref{eq11}}\leq Ce^{ -2\sqrt{\lambda_1}R},
\]
as desired. Then \eqref{eq13} is a direct consequence of \eqref{eq12}.
\end{proof}
These estimates allow us to deduce the following important dichotomy result. In proving it we shall need some technical tool about calculus on $\RCD$ spaces. In particular, we shall use the   \textit{Gauss-Green formula} proved in  \cite{BPS19} (that in turn relies on the machinery developed in \cite{DGP19} to speak about trace of a Sobolev vector field at the boundary of a set). We shall also need the coarea formula, that is well known to be valid even in the metric setting, see  \cite{BPS19}, \cite{BCM22} and references therein (starting from the original \cite{Mir03}, where the first instance of the coarea formula in the metric setting has been obtained).

%
%
%
\begin{thm}\label{thm14}
Let \((\X,\di,\m)\) be an \(\RCD(K,N)\) m.m.s.\ and let \(E\) be an end of \(\X\). Assume that \(\lambda_1>0\). Then exactly one of the following holds:
\begin{enumerate}
	\item There exists a bounded non-constant harmonic function on \(E\) and \(E\) has exponential volume growth. More precisely:
	\begin{equation}\label{stimavolumiinfiniti}
	V_E(R)\geq C\exp\left(2\sqrt{\lambda_1}R\right)
	\end{equation}		
	for every \(R\) large enough and some constant \(C>0\).
		\item Every bounded harmonic function on \(E\) is constant and \(E\) has exponential volume decay. More precisely:
	\begin{equation}\label{stimavolumifiniti}
	V_E(\infty)-V_E(R)\leq C\exp\left(-2\sqrt{\lambda_1}R\right)
	\end{equation}		
		for every \(R\) large enough and some constant \(C>0\).
\end{enumerate}
\end{thm}
\begin{proof}\ \\ 
\textit{Step 1.} Assume that every bounded harmonic function on \(E\) is constant. Then this is the case for the function $f_E$ given by Definition \ref{def:fE}. We have then that \(f_E\) is equal to \(1\) a.e. on \(E\), and by \((\ref{eq11})\) follows that
\[V_E(R+1)-V_E(R)\leq C\exp\left(-2\sqrt{\lambda_1}R\right).\]
Summing up we obtain
\[V_E(\infty)-V_E(R)\leq C\sum_{i=0}^{+\infty}\exp\left(-2\sqrt{\lambda_1}(R+i)\right)\leq C\exp\left(-2\sqrt{\lambda_1}R\right).\]

\textit{Step 2.} Assume now that there exists a bounded non-constant harmonic function $u$ on \(E\). By translation and rescaling we can assume that $u\geq 0$ on $\X$, $u\geq 1$ on $K'$ and $u<1$ somewhere in $E$ (recall \eqref{eq:kp}). Then by the maximum principle and the construction of $f_E$ it follows that $f_E$ is not constant as well.

Now notice that by the spherical version of the Bishop-Gromov inequality we have that $B_R(p)$ has finite perimeter for any $R>0$. It follows that for any $R>R_0$ the set $E^R:=E\setminus E(R)$ has finite perimeter with $\Per(E^R,\cdot)$ being the restriction of $\Per( {B_R(p)},\cdot)$ to $E$.

We claim that there is $c>0$ such that
\begin{equation}
\label{eq:claimper}
\int |\nabla f_E|\, \d\Per(E^r,\cdot)\geq c\qquad\forall r>R_0,
\end{equation}
where here, with a slight abuse of notation, we are denoting by $\nabla f_E\in L^0(T \X,\Per(E^r,\cdot))$ the trace of  $\nabla f_E\in W^{1,2}_{C,loc}(E)$ in the sense of \cite{DGP19} and \cite{BPS19}. 

Let us show how from \eqref{eq:claimper} we can conclude. Putting $P(r):=\Per(E^r,\X)$, by \eqref{eq:claimper} and the Cauchy-Schwarz inequality we immediately have   $\tfrac{1}{P(r)}\leq \frac1c \int |\nabla f_E|^2\, \d\Per(E^r,\cdot)$ and therefore
\begin{equation}
\label{eq:calcolocontraccia}
\begin{split}
1&\leq \int_R^{R+1}P(r)\,\d r \int_R^{R+1}\tfrac1{P(r)}\,\d r\\
&\stackrel*\leq\tfrac1c\big( V_E(R+1)-V_E(R)\big)\int_{E(R+1)\setminus E(R)} |\nabla f_E|^2\,\d \m\\
\text{(by \eqref{eq12})}\qquad&\leq C \big( V_E(R+1)-V_E(R)\big)e^{-2R\sqrt{\lambda_1}},
\end{split}
\end{equation}
where in the starred inequality we used the coarea formula to deduce $V_E(R+1)-V_E(R)=\m(E(R+1)\setminus E(R))=\int_{E(R+1)\setminus E(R)}|\nabla {\sf d}_p|\, \d\m=\int_R^{R+1}P(r)\,\d r$ and the fact that, if we call $\tilde\nabla f_E\in L^0_{Cap}(T\X)$ the quasi-continuous representative of $\nabla f_E\in W^{1,2}_{C,loc}(E)$ (set to 0 outside $E$, say) and $\tr_r \nabla f_E$ the trace of $\nabla f_E$ in $L^0(T \X,\Per(E^r,\cdot))$, then, by definition,  $\tr_r \nabla f_E$ is the equivalence class of $\tilde\nabla f_E$ in  in $L^0(T \X,\Per(E^r,\cdot))$ and thus we have $|\tr_r \nabla f_E|=|\tilde\nabla f_E|$ $\Per(E^r,\cdot)$-a.e.. This and the trivial identity $|\tilde\nabla f_E|=|\nabla f_E|$ $\m$-a.e.\ justifies the computation above. 

Since \eqref{eq:calcolocontraccia} is equivalent to the claim, we are left to prove \eqref{eq:claimper}. 
Since \(f_E\) is harmonic, putting $F(r):= \int\langle \nabla f_E,\nu_{E(r)}\rangle \,\d\Per(E^r,\cdot)$ for brevity,  by the Gauss-Green formula we have that, for every \(r_2>r_1>R_0\) it holds
\[
\begin{split}
F(r_2)-F(r_1)=\int  \langle\nabla f_E,\nu_{E(r_2)\setminus E(r_1)}\rangle \,\d\Per(E(r_2)\setminus E(r_1),\cdot)=-\int_{E(r_2)\setminus E(r_1)}\Delta f_E\,\d\mm=0,
\end{split}
\]
thus $r\mapsto F(r)$ is constant on $(R_0,+\infty)$ and since clearly $|F(r)| \leq \int| \nabla f_E| \d\Per(E^r,\cdot)$, to get  \eqref{eq:claimper} it suffices to prove that $F(r)\neq 0$  for some $r>R_0$.

To see this, fix $\bar R>R_0$ and notice that by the strong maximum principle and the fact that $f_E$ is not constant, there must be $ a>0$ such that  $\sup_{\partial E^{\bar R} }f_E\leq 1-a$. In particular, taking into account the continuity of $f_E$ and  the Sobolev-to-Lipschitz property it easily follows that there is $b\in(0,\tfrac { a}2)$ such that
\begin{equation}
\label{eq:gradnonzero}
\int_{E\cap B_{\bar R}(p)\cap\{ 1-2b<f_E<1-b\}}|\nabla f_E|^2\,\d\m>0.
\end{equation}
Put $\rho:=\varphi\circ f_E$ where
\[
\varphi(z):=\left\{\begin{array}{l}
1,\quad\text{ if }z\leq 1-2b,\\
0,\quad\text{ if }z\geq 1-b,\\
\text{affine and continuous on  }[1-2b,1-b].\\
\end{array}
\right.
\]
and notice that we have
\[
\int_{E(\bar R)}\divv(\rho\nabla f_E)\,\d\m=\int_{E(\bar R)}\nabla\rho\cdot\nabla f_E\,\d\m=-\frac1b
\int_{E\cap B_{\bar R}(p)\cap\{ 1-2b<f_E<1-b\}}|\nabla f_E|^2\,\d\m\stackrel{\eqref{eq:gradnonzero}}<0.
\]
We now claim that there is a finite perimeter set $\tilde E$ such that $\tilde E\Delta E(\bar R)\subset f^{-1}_E((1-b,1))$. Assuming we have such $\tilde E$, we conclude with
\[
\begin{split}
0>\int_{E(\bar R)}\divv(\rho\nabla f_E)\,\d\m=\int_{\tilde E}\divv(\rho\nabla f_E)\,\d\m=\int \rho\langle \nabla f_E,\nu_{\tilde E}\rangle\,\d\Per(\tilde E,\cdot)=-F(\bar R),
\end{split}
\]
where in the last step we used the fact that $\rho\equiv 1$  on $\partial E^{\bar R}=\partial\tilde E\cap f_E^{-1}((0,1-b])$ and $\rho\equiv0$ on $\partial\tilde E\setminus f_E^{-1}((0,1-b]).$ It thus remains to show the existence of such $\tilde E$. To see this, recall that balls have finite perimeter (by the spherical version \eqref{eq:BG} of the Bishop-Gromov inequality) and that $\partial E$ is compact. Then using the continuity of $f_E$ we can find  $r>0$ such that for any $x\in\partial E$ we have $B_r(x)\subset \{f_E\geq 1-b\}$ and by compactness a finite number of points $x_1,\ldots,x_n$ such that $\partial E\subset \cup_iB_r(x_i)$. It is then clear that the set $\tilde E:=E\setminus  \cup_iB_r(x_i)$ does the job.
\end{proof}
As a direct consequence of the above result we obtain the following:
\begin{cor}\label{thm2ends}
Let \((\X,\di,\m)\) be an \(\RCD(K,N)\) space with \(\lambda_1>0\) and with at  least two ends with infinite volume. 

Then there exists a bounded non-constant harmonic function on \(\X\).
\end{cor}
\begin{proof}  Let $E_1,E_2$ be the two given ends with infinite volume and   consider the associated functions \(f_{E_1},f_{E_2}\) as in Definition \ref{def:fE} (possibly enlarging $R_0$ we can take $K':=B_{R_0}(p)$), which by {Theorem} \ref{thm14} are bounded non-constant and harmonic functions. Recall that   \(f_{E_i}\equiv 1\) on \(\partial E_i\) and  \(\text{ess}\inf f_{E_i}=0\) for \(i=1,2\).

Then fix $R>R_0+1$ and let $C_i:=\{x\in B_R(p):{\sf d}(x,E_i\cap\partial B_{R+1}(p))\geq 1\}$. Then a simple variant of  Theorem \ref{thm:exharm} gives the existence of a function $f_R\in (W^{1,2}_0\cap C)(B_{R+2}(p))$ with values in $[0,1]$ that is harmonic in $B_{R+2}(p)\setminus (C_1\cup C_2)\supset B_R(p)$, equal to 1 on $C_1$ and equal to 0 on $C_2$. 

We claim that on $B_R\cap E_2$ we have $f_R\leq f_{E_2}$. To see this notice that $\partial(B_R\cap E_2)$ is contained in the disjoint union of $\partial E_2$ and $C_2$ and by construction we have $f_{E_2}=1$ on $\partial E_2$  and $f_R=0$ on $C_2$. Since both functions have values in $[0,1]$, are harmonic in  $B_R\cap E_2$ and continuous up to the boundary of such set, by the maximum principle our claim follows. 

Analogously we can prove that $f_R\geq 1-f_{E_1}$ on $B_R\cap E_1$. We now want to send $R\uparrow+\infty$ and find, possibly after passing to a subsequence, a limit harmonic function on the whole $X$. This is possibile thanks to the Lipschitz estimates  \eqref{chengyau}, that grant that for some $C(R)$ the Lipschitz constant of $f_{R'}$ on $B_{R/2}(p)$ is bounded from above by $C(R)$ for any $R'>R$. By Arzel\'a-Ascoli's theorem, this suffices to find a limit function $f$ and it is then easy to see that this function is harmonic and bounded (in fact with values in $[0,1]$).

Also, a direct consequence of the construction and of the previous claims is that $f\leq f_{E_2}$ on $E_2$ and $f\geq 1-f_{E_1}$ on $E_1$. If we knew that $\inf_{E_i}f_{E_i}=0$ for $i=1,2$ the conclusion would directly follow, as $\inf f\leq \inf_{E_2}f_{E_2}=0$ and $\sup f\geq \sup_{E_1}1-f_{E_1}=1-\inf_{E_1}f_{E_1}=1$, proving that $f$ is not constant. 

Thus it remains to show that $\inf_{E_1}f_{E_1}=0$ (the argument for $E_2,f_{E_2}$ being analogous). Say not. Then by construction $a:=\inf_{E_1}f_{E_1}>0$. Define the new function $\tilde f_1:=1-\tfrac{1-f_{E_1}}{1-a}$ and notice that it is still harmonic and continuous on $E_1$ with boundary value equal to $1$ and that it is still positive. Thus by the maximum principle the function $\tilde f_1$ bounds from above each of the harmonic functions defined on $B_R(p)\cap E_1$ that are used in the definition of $f_{E_1}$. It follows by the definition of $f_{E_1}$ that it would hold $f_{E_1}\leq \tilde f_1$. This, however, contradicts the definition of $\tilde f_1$, as this ensures that $\tilde f_1<f_{E_1}$ on $E_1$.
\end{proof}

This last corollary should be coupled with the following simple and general result, stating that if $\lambda_1>0$ (an assumption that is present in Theorem \ref{mainthm}), then the space has at least one end of infinite volume:
\begin{prop}\label{prop:1end}
Let \((\X,\di,\m)\) be a metric measure space with  \(\lambda_1>0\). Then \(\X\) has at least one end with infinite volume.
\end{prop}
\begin{proof}
By the definition of \(\lambda_1\) it follows immediately that if \(\m(\X)<\infty\) then \(\lambda_1=0\), this means that \(\m(\X)=\infty\).

Assume by contradiction that \(\X\) has no ends with infinite volume. Then there exists a compact \(K\) which has positive volume and cuts the space in ends with finite volume. Let \(\{E_i\}_{i\in\N}\) be the set of the ends with respect to \(K\). Since \(\m(E_i)<\infty\) for every \(i\in\N\) then for every \(\varepsilon>0\) there exist a couple of radii \(r_{i,\varepsilon}\), \(R_{i,\varepsilon}>r_{i,\varepsilon}+2\) such that \(V_{E_i}(R_{i,\varepsilon})- V_{E_i}(r_{i,\varepsilon})\leq\frac{\varepsilon}{2^i}\). Taking now the function 
\[
f_\varepsilon(x):=\begin{cases} 1 & \mbox{for } x\in K\cup\left(\bigcup(E_i(r_{i,\varepsilon}))\right), \\ \frac{\di(x,p)-r_{i,\varepsilon}}{R_{i,\varepsilon}-r_{i,\varepsilon}} & \mbox{for } x\in E_i(R_{i,\varepsilon})\setminus E_i(r_{i,\varepsilon}), \\ 0 & \mbox{for } x\in E_i\setminus E_i(R_{i,\varepsilon}),\end{cases}\]
we conclude that
\[\frac{\int_X |\nabla f_\varepsilon|^2 \,\d\m}{\int_X|f_\varepsilon|^2\,\d\m}\leq \frac{\varepsilon}{\m(K)}\xrightarrow[\varepsilon\rightarrow 0]{}0.\]
\end{proof}

\subsection{Properties of the bounded harmonic  function}\label{chbusemann}
Let us fix notations and assumptions we are going to use for a bit:
\begin{assumption}\label{baseas}
$(\X,\sfd,\mm)$ is an $\RCD(-(N-1),N)$ space with \(N\geq 3\) and \(\lambda_1\geq N-2\). We also assume that $\supp(\mm)=\X$ and that $\X$ has two ends with infinite volume.

Finally, we shall denote by $u$ a fixed bounded and non-constant harmonic function on $\X$. The existence of such $u$ is granted by \textit{Corollary} \ref{thm2ends}.
\end{assumption}
We start with the following simple regularity statement:
\begin{lemma}\label{le:baseb}
With the same assumptions and notation as in Assumption \ref{baseas} the following holds.

The function $u$ is in $\Test_{loc}(\X)$ and globally Lipschitz.
\end{lemma}
\begin{proof}
By the Cheng-Yau gradient estimate \((\ref{chengyau})\) and the fact that $u$ is bounded it follows that it is globally Lipschitz. This information together with the fact that $\Delta u\equiv0$ suffices to ensures that $u\in \Test_{loc}(\X)$.
\end{proof}

The rigidity result in this paper is ultimately a consequence, as customary, of the equality in the Bochner inequality. Before coming to that it is useful to recall that following result, that in the $\RCD$ setting has been proved in \cite{GV21}. We provide anyway the complete proof because we shall be interested in the equality case, that was not explicitly studied in \cite{GV21}:
\begin{lemma}[Generalized refined Kato inequality]\label{le:kato} Let  \((\X,\sfd,\mm)\)  be a   \(\RCD(K,N)\) space with \(K\in \R\) and \(N\in[1,+\infty)\). Then for any \(u\in H^{2,2}_{loc}(\X)\) it holds 
\begin{equation}\label{kato}
\frac{t+\dim(\X)}{t+\dim(\X)-1}|\d|\d u||^2\leq|\Hess(u)|^2_{\sf HS}+\frac{\left(\tr\Hess(u)\right)^2}{t}\quad\mm-a.e.\qquad \forall t>0,
\end{equation}
where $\dim(\X)\in \N\cap [1,N]$ is the (constant) dimension of $L^0(T\X)$, recall Theorem \ref{thm:dimX}. Equality holds $\mm$-a.e.\ for a given $t$ if and only if 
\begin{equation}
\label{eq:eqhess}
\Hess (u)=\alpha\big(\Id-(t+\dim(\X))\,\e_1\otimes\e_1\big),
\end{equation}
for some  function $\alpha\in L^2_{loc}(\X)$, where $\e_1\in L^0(T\X)$ is a pointwise unitary vector that is equal to $\tfrac{\nabla u}{|\nabla u|}$ on  $|\nabla u|>0$.
\end{lemma}
\begin{proof} Let \(A\) be a symmetric \(n\times n\) real matrix. We claim that  for every \(t>0\) and  every \(v\in \R^n\) we have:
\begin{equation}
\label{eq:claimmatrix}
\frac{t+n}{t+n-1}|A\cdot v|^2\leq|v|^2|A|^2_{\sf HS}+ \frac{|v|^2(\tr A)^2}{t}.
\end{equation}
Indeed, by the spectral theorem there is an orthonormal base $\e_1,\ldots,\e_n$ of $\R^n$ so that  \(A\) is diagonal w.r.t.\ such base, with diagonal entries \(\alpha_1,...,\alpha_n\). We can also assume that   \(|\alpha_1|\geq |\alpha_i|\) for \(i=2,...,n\). Then, applying twice the Cauchy-Schwarz inequality we obtain
\begin{align*}
&|v|^2\left( \frac{(\alpha_1+\alpha_2+...+\alpha_n)^2}{t}+\alpha_1^2+\alpha_2^2+...+\alpha_n^2\right)\\
&\qquad\stackrel{(1)}\geq |v|^2\left(\frac{(\alpha_1+\alpha_2+...+\alpha_n)^2}{t}+\frac{(\alpha_2+\alpha_3+...+\alpha_{n})^2}{n-1} +\alpha_1^2 \right)\\
&\qquad\stackrel{(2)}\geq |v|^2\left( \frac{\alpha_1^2}{t+n-1} +\alpha_1^2\right)\\
&\qquad\stackrel{(3)}\geq \frac{t+n}{t+n-1}|A\cdot v|^2,
\end{align*}
which is \eqref{eq:claimmatrix}. Equality holds in \eqref{eq:claimmatrix} iff the inequalities above are all equality. Equality in $(1)$ holds iff $\alpha_2=\cdots=\alpha_n=:\alpha$,  equality in $(2)$ iff $\alpha=\alpha_1(1-(t+n))$. Finally, equality in $(3)$ holds iff $v$ is a multiple of $\e_1$. In other words, equality in \eqref{eq:claimmatrix} holds iff 
\[
A=\alpha(\Id-(t+n)\e_1\otimes\e_1).
\]
We come to \eqref{kato} and its equality case and start noticing that  $\d|\d u|=\nchi_{\{|\nabla u|>0\}}\Hess u(\tfrac{\nabla u}{|\nabla u|})$ (for $u\in\Test_{loc}(\X)$ this follows from \cite[Proposition 3.3.32 - (i)]{Gigli14} and the chain rule - see also \cite[Lemma 2.5]{DGP19}, then the general case comes by approximation). Now the conclusion follows picking a suitable pointwise orthonormal base $\e_1,\ldots,\e_{\dim(\X)}\in L^0(T\X)$ of $L^0(T\X)$ that diagonalizes $\Hess u$.\end{proof}

\begin{prop}\label{hessianbusemann} With the same assumptions and notation as in  Assumption \ref{baseas} the following hold:
\begin{enumerate}
	\item \(\lambda_1=N-2\).
	\item $|\d u|$ is locally Lipschitz (i.e.\ it has a locally Lipschitz representative) and \emph{strictly positive}.
	\item Putting $\e_1:=\frac{\nabla u}{|\nabla u|}$ (this is well defined by the above) we have
	\begin{equation}
\label{eq:hessbus}
\Hess(u)=\alpha  (\Id-N\e_1\otimes\e_1),
\end{equation}
for some $\alpha\in L^2_{loc}(\X)$.
\item We have $|\d u|^{\frac{N-2}{N-1}}\in D(\Delta_{loc})$ with
\begin{equation}
\label{eq:rigdelta}
\Delta (|\d u|^{\frac{N-2}{N-1}})=-(N-2)|\d u|^{\frac{N-2}{N-1}}.
\end{equation}
\item $u$ is an open map, i.e.\ $u(U)\subset \R$ is open for any $U\subset\X$ open.
\end{enumerate}
\end{prop}
\begin{proof} The required rigidity will follow by closely inspecting the proof of \cite[Theorem 3.4]{GV21}. Start noticing that since we know from Lemma \ref{le:baseb} above that $u\in \Test_{loc}(\X)$, we can apply the improved Bochner inequality \ref{bochner} (say $N>\dim(\X)$; the case $N=\dim(\X)$ follows along similar lines recalling that in this case $\tr\Hess u=\Delta u=0$ and that the last addend in the Bochner inequality below is 0 - see \cite{Han14}) recalling that $\Delta u=0$ to get
\[
\bd\tfrac{|\d u|^2}{2}\geq\Big(|\Hess u|_{\sf HS}^2-(N-1)|\d u|^2+\frac{(\tr\Hess u)^2}{N-\dim(\X)}\Big)\m.
\]
Kato's inequality as in Lemma \ref{kato} with $t=N-\dim(\X)$ can be written as 
\[
|\Hess u|_{\sf HS}^2=\frac N{N-1}|\d|\d u||^2-\frac{(\tr\Hess u)^2}{N-\dim(\X)}+F,\qquad \text{for some }F\geq 0\ \m-a.e.
\]
and the equality case reads as
\begin{equation}
\label{eq:eqkato}
\text{$F=0$ $\m$-a.e.}\qquad\Rightarrow\qquad \text{identity \eqref{eq:hessbus}  holds.}
\end{equation}
We thus have
\begin{equation}
\label{eq:conF}
\bd\tfrac{|\d u|^2}{2}\geq (\tfrac{N}{N-1}|\nabla|\nabla u||^2-(N-1)|\nabla u|^2 +F)\m
\end{equation}
Now let us fix  $\beta,\eps>0$ and put $\varphi(z):=\varphi_{\beta,\eps}(z):=(z+\eps)^\beta$.  Then \eqref{eq:conF} and basic calculus rules (see also the proof of  \cite[Theorem 3.4]{GV21}) give
\[
\begin{split}
\bd(\varphi(|\d u|^2))\geq \Big(\beta(|\d u|^2+\eps)^{\beta-1}(F-2(N-1)|\d u|^2)+|\d|\d u||^2\Big(\beta\tfrac{2N}{N-1}+\tfrac{4\beta(\beta-1)|\d u|^2}{|\d u|^2+\eps}\Big)\Big)\m.
\end{split}
\]
Now pick  $\beta:=\frac{N-2}{2(N-1)}$ and let $\eps\downarrow0$: the rightmost addend goes to 0 and putting for brevity $f:=|\nabla u|^{\frac{N-2}{N-1}}$ we are left with
\begin{equation}
\label{eq:perrigid}
\bd f+\lambda_1f\m \geq \big( {(\lambda_1-(N-2))}f+\tfrac{N-2}{2(N-1)} F|\d u|^{-\frac{N}{N-1}}\big)\m\geq 0
\end{equation}
(notice that the passage to the limit is justified also by the fact that $f\in W^{1,2}_{loc}(\X)$, as established in  \cite[Theorem 3.4]{GV21} - the proof of this  follows from the computations we are repeating here). 

To deduce the desired rigidity from the above, start observing  that for any $\varphi\in \Lip_{bs}(\X)$ we have
\[
\int \varphi^2 f\,\d\bd f=\int-\varphi^2|\d f|^2-2f\varphi\langle\d\varphi,\d f\rangle\,\d\m=\int-|\d(\varphi f)|^2+f^2|\d \varphi|^2\,\d\m,
\]
and thus using that $\lambda_1\int f^2\varphi^2\,\d\mm\leq \int |\d(\varphi f)|^2\,\d\mm$ we get
\[
\int \varphi^2f\,\d(\bd f+\lambda_1f\m)\leq \int f^2|\d\varphi|^2\,\d\m.
\]
Notice that \eqref{eq:perrigid} and the assumption $\lambda_1\geq N-2$ tell in particular that $\bd f+\lambda_1f\m\geq 0$, then fix $R>0$, let $\varphi:=(1-R^{-1}\sfd(\cdot,B_R(p)))^+$ in the above then let $R\uparrow\infty$ to get
\begin{equation}
\label{eq:rigid2}
\int f\,\d(\bd f+\lambda_1f\m)\leq\liminf_{R\to+\infty}\frac1{R^2}\int_{B_{2R}(p)\setminus B_R(p)}f^2\,\d\m.
\end{equation}
Let us prove that the right hand side is zero. Putting $A_R:=B_{2R}(p)\setminus B_R(p)$, from Holder's inequality we get
\[
\begin{split}
\int_{A_R}f^2\,\d\m&\leq\Big(\int_{A_R}|\d u|^2 e^{2\sqrt{\lambda_1}\sfd(\cdot,p)}\,\d\m\Big)^{\frac{N-2}{N-1}}\Big(\int_{A_R} e^{-2(N-2)\sqrt{\lambda_1}\sfd(\cdot,p)}\,\d\m\Big)^{\frac1{N-1}}\\
\text{(by \eqref{eq13})}\qquad&\leq  CR^{\frac{N-2}{N-1}}\Big(\int_R^{2R} e^{-2(N-2)\sqrt{\lambda_1}r}s(r)\,\d \m\Big)^{\frac1{N-1}}\\
\text{(by \eqref{eq:BG})}\qquad &\leq  CR^{\frac{N-2}{N-1}}\Big(\int_R^{2R} e^{(-2(N-2)\sqrt{\lambda_1}+N-1)r}\,\d \m\Big)^{\frac1{N-1}}\leq CR,
\end{split}
\]
where in the last inequality we used the fact that $-2(N-2)\sqrt{\lambda_1}+N-1\leq 0$, that in turn follows from $\lambda_1\geq (N-2)$ and $N\geq 3$.

Therefore from \eqref{eq:rigid2} we see that $\int f\,\d(\bd f+\lambda_1f\m)\leq0$ and thus from  \eqref{eq:perrigid} we conclude that
\[
(\lambda_1-(N-2))\int f^2\,\d\m+\int |\nabla u|^{-\frac{2}{N-1}}F\, \d\m\leq 0.
\]
Since $\lambda_1\geq N-2$ and $f$ is not identically 0 (as $u$ is not constant), we deduce that $\lambda_1=N-2$, i.e.\ item (1) holds. Also, since $ |\nabla u|^{-\frac{2}{N-1}}\geq \Lip(u)^{-\frac{2}{N-1}}>0$, we see that $F=0$ $\m$-a.e., and thus by \eqref{eq:eqkato} item (3) holds as well.

Now notice that we now know that 
\begin{equation}
\label{eq:feigen}
\bd f=-(N-2)f\m
\end{equation}
and by the compatibility of the concepts of measure-valued and $L^2$-valued Laplacian (see \cite[Chapter 4]{Gigli12}), item $(4)$ follows. Moreover, we also deduce that $\bd f\leq 0$, hence according to \cite{Gigli-Mondino12} $f$ is superharmonic on any bounded open subset of $\X$. Since $\X$ is locally doubling and supports a local weak Poincar\'e inequality, the conclusion follows from the weak Harnack inequality (see e.g.\ \cite[Chapter 8]{Bjorn-Bjorn11}) recalling, once again, that $f$ is not identically 0. Even more, combining \eqref{eq:feigen} and the Bochner inequality it is not hard to see that 
\[
\begin{split}
\bd\tfrac{|\d f|^2}2\geq \la\d f,\d\Delta f\ra-(N-1)|\d f|^2=-(2N-3)|\d f|^2,
\end{split}
\]
thus the same arguments just expressed tell that $|\d f|$ is locally bounded from above, i.e.\ that $f$ is locally Lipschitz (by the Sobolev-to-Lipschitz property). Hence the same holds for $|\d u|=f^{\frac{N-1}{N-2}}$.

It remains to prove that $u$ is open. Fix $\bar x\in\X$, let $\xi:\X\to[0,1]$ be a Lipschitz cut-off function with compact support identically 1 on $B_2(\bar x)$. Let $(\Fl_t)$ be the Regular Lagrangian Flow of $\xi\nabla u$, that is easily seen to exists as $|\xi\nabla u|\in L^\infty$, $\div(\xi\nabla u)=\la\nabla \xi,\nabla u\ra\in L^\infty$ and $\xi\nabla u\in W^{1,2}_C(T\X)$ by \cite[Proposition 3.3.22]{Gigli14}. By the defining property \eqref{RLFder} and the finite speed of propagation that follows from \eqref{eq:speedrlf}, we see that there is $T>0$ such that $\partial_t u(\Fl_t(x))=|\d u|^2(\Fl_t(x))$ holds for $\mm$-a.e.\ $x\in B_1(\bar x)$ and $t\in[-T,T]$. For these $x,t$, recalling also item $(2)$ we see that $\partial_tu(\Fl_t(x))\geq c$ for $c:=\inf_{B_2(\bar x)}|\d u|^2>0$. Fix $x$ for which this holds for a.e.\ $t\in[-T,T]$ and use the continuity of $u$ to deduce that the image under $u$ of $\{\Fl_t(x):t\in[-T,T]\}$ contains $[u(x)-Tc,u(x)+Tc]$. Picking $x$ sufficiently close to $\bar x$ we conclude that $u(B_1(\bar x))$ contains a neighbourhood of $u(x)$ and repeating the argument with $B_r(\bar x)$, $r\ll1$, in place of $B_1(\bar x)$ we conclude. 
\end{proof}

\subsection{$|\d u|$ as a function of $u$}\label{chcompfun}

In this section we prove that the weak upper gradient of \(u\) is of the form $\varphi\circ u$ for a suitable smooth function $\varphi$. Notice that from \eqref{eq:hessbus} it follows that
\[
\d|\d u|=\Hess u(\tfrac{\nabla u}{|\d u|})=\alpha(1-N)\tfrac{\d u}{|\d u|},
\]
i.e.\ $\d|\d u|=h\d u$ for some function $h$. In the smooth category, this suffices to conclude that, locally, $|\d u|$ is a function of $u$, a natural line of thought being: let $U$ be a small open set such that level sets of $u$ are smooth-path-connected in $U$ (recall that $|\d u|>0$), then pick $x,y\in U$ with $u(x)=u(y)$ and find a smooth curve $\gamma$ joining them, with values in $U$ and with $t\mapsto u(\gamma_t)$ constant. Then $0=\partial_t u(\gamma_t)=\d u(\gamma_t')$ and thus  $\partial_t(|\d u|(\gamma_t))=\d|\d u|(\gamma_t')=h(\gamma_t)\d u(\gamma_t')=0$, proving that $|\d u|$ is also constant along $\gamma$ and thus that, on $U$, the value of $|\d u|$ depends solely on that of $u$.

We are going to prove an analogous statement in our setting: roughly said, the underlying idea is the same just exposed, but the technicalities are much more involved. 
\begin{prop}\label{prop:comploc}
Let $(\X,\sfd,\mm)$ be an $\RCD(K,N)$ space and $f,g\in \Lip_{loc}(\X)$ be such that $\d g=h \d f$ for some $h:\X\to\R$. 

Also, assume that $f\in D(\Delta_{loc})\cap W^{2,2}_{loc}(\X) $ is with  $\frac1{|\d f|},|\Delta f|,|\Hess f|_{\sf HS}\in L^\infty_{loc}(\X)$.

Then for every $x\in \X$ there are a neighbourhood $U$ and a Lipschitz function $\varphi:\R\to\R$ such that $g=\varphi\circ f$ on $U$.
\end{prop}
\begin{proof}
Fix $\bar x\in \X$. 
Let  $\eta:\X\to[0,1]$ be Lipschitz, with bounded support  and identically 1 on $B_{3 }(\bar x)$. Then by direct computation we see that for the vector field $v:=\eta\tfrac{\nabla f}{|\nabla f|^2}$ we have
\[
\begin{split}
\divv(v)&=\eta\frac{\Delta f}{|\nabla f|^2}-2\eta\frac{\langle \nabla|\nabla f|,\nabla f\rangle}{|\nabla f|^3}+\frac{\left\langle \nabla \eta,\nabla f\right\rangle}{|\nabla f|^2}\\
\nabla v&=\eta\frac{\Hess(f)}{|\nabla f|^2}-2\eta\frac{\nabla|\nabla f|\otimes \nabla f}{|\nabla f|^3}+\frac{\nabla\eta\otimes\nabla f}{|\nabla f|^2}
\end{split}
\]
so that our assumptions on $f,\eta$ grant that $v$ has bounded divergence and bounded covariant derivative. It follows from \cite{BS18a} that the Regular Lagrangian Flow $\Fl:\R\times\X\to\X$ of $v$ is Lipschitz in space and time. By definition of Regular Lagrangian Flow and of $v$ the equation $\partial_t(f(\Fl_t(x)))=\eta(\Fl_t(x))$ holds for $\mm$-a.e.\ $x$ and a.e.\ $t$, but thanks to the continuity of $\eta,\Fl,f$ it is easy to see that in fact for any $x\in\X$ the map $t\mapsto f(\Fl_t(x))$ is $C^1$ with derivative equal to $\eta(\Fl_t(x))$.

Now consider  the map $H:\X\times\R\times\R\to\X$ defined as $H(x,t,s):=\Fl_{t+s-f(x)}(x)$. Then this is clearly Lipschitz. Also, let $B\subset\X$ be bounded and $I\subset\R$ be a bounded interval. We claim that for some $C>0$ we have
\begin{equation}
\label{eq:Hcomp}
\text{$H_*(\mm\restr B\times\delta_{t} \times\mathcal L^1\restr I)\leq C\mm\qquad\forall t\in I.$}
\end{equation}
Indeed, since $f$ is Lipschitz there is $T>0$ such that $|t+s-f(x)|\leq T$ holds for any $(x,t,s)\in B\times I\times I$. It follows that for any $\varphi\in C_b(\X)$ non-negative it holds
\[
\begin{split}
\int\varphi\,\d H_*(\mm\restr B\times\delta_t \times\mathcal L^1\restr I)&=\int_B\int_I\varphi(\Fl_{t+s-f(x)})\,\d s\,\d\mm(x)\\
&\leq \int_B\int_{-T}^T\varphi(\Fl_r)\,\d r\,\d\mm(x)\leq 2T{\sf Comp}(\Fl)\int\varphi\,\d\mm,
\end{split}
\]
thus proving our claim \eqref{eq:Hcomp}. Notice also that since $|v|\in L^\infty$, the estimate \eqref{eq:speedrlf} trivially ensures that
\begin{equation}
\label{eq:claimH}
\text{there is $\bar t>0$ so that for any $x\in B_{2}(\bar x)$ we have $\Fl_t(x)\in B_{3}(\bar x)$ for any $t\in[-\bar t,\bar t]$.}
\end{equation}
Put  $\bar r:=\min\{\tfrac12,\tfrac{\bar t}8(\Lip(f\restr{B_3(\bar x)}))^{-1}\}$. We  claim that
\begin{equation}
\label{eq:claim1}
\text{there is $\varphi:\R\to\R$ such that $g=\varphi\circ f$ on $B_{\bar r}(\bar x)$}
\end{equation}
and argue by contradiction. If not, there are $x_0,x_1\in B_{\bar r}(\bar x)$ with $f(x_0)=f(x_1)$ and $g(x_0)<g(x_1)$.  By continuity we can find neighbourhoods $U_0,U_1\subset B_{\bar r}(\bar x)$ of $x_0,x_1$ respectively such that $\inf_{U_1}g>\sup_{U_0}g$. Then we can find other neighbourhoods $V_i\subset U_i$ of $x_i$, $i=0,1$, and $\bar s\in(0,\tfrac{\bar t}2)$ such that  $\Fl_{s}(x)\in U_i$ for every $x\in V_i$, $s\in[-\bar s,\bar s]$, $i=0,1$.

Let $\mu_i:=\mm(V_i)^{-1}\mm\restr{V_i}$, $i=0,1$ and let $\ppi$ be the only (by \cite{GigliRajalaSturm13}) optimal geodesic plan joining $\mu_0$ to $\mu_1$. Then \cite{Rajala12-2}  ensures that $\ppi$ is a test plan and clearly it is concentrated on geodesics taking values in $B_{2\bar r}(\bar x)\subset B_1(\bar x)$. Also, let $G:C([0,1],\X)\times\R\to C([0,1],\X)$ be given by $G(\gamma,s)_t:=H(\gamma_t,f(x_0),s)$ and put $\hat\ppi:=G_*(\ppi\times(\tfrac1{2\bar s}\mathcal L^1\restr{[-\bar s,\bar s]}))$. From the fact that $H$ is Lipschitz and \eqref{eq:Hcomp} it directly follows that $\hat\ppi$ is a test plan as well and the construction also ensures that $(\e_i)_*\hat\ppi$ is concentrated on $V_i$, $i=0,1$.

We make now the intermediate claim: 
\begin{equation}
\label{eq:hatppiconst}
\text{$\hat\ppi$ is concentrated on curves along which $f$ is constant. }
\end{equation}
To see this, notice that $\ppi$ is concentrated on geodesics having endpoints in $B_{\bar r}(\bar x)$ and thus that it is sufficient to prove that for any such geodesic $\gamma$, any $t\in[0,1]$ and $s\in[-\bar s,\bar s]$ we have
\[
f(\Fl_{f(x_0)+s-f(\gamma_0)}(\gamma_t))=f(x_0)+s.
\]
Since $r\mapsto f(\Fl_r(\gamma_t))$ is $C^1$ with derivative $\eta(\Fl_r(\gamma_t))$, the claim follows if we show that $\Fl_r(\gamma_t)\in B_3(\bar x)\subset \{\eta=1\}$ for any $|r|\leq |f(x_0)+s-f(\gamma_0)|$. To see this, notice that $\gamma$  takes values in $B_{2\bar r}(\bar x)\subset B_2(\bar x)$ and thus $|f(\gamma_0)+s-f(\gamma_t)|\leq \bar s+\Lip (f\restr{B_3(\bar x)})\sfd(\gamma_0,\gamma_t)\leq \bar t$  for  any $t\in[0,1]$ and  $s\in[-\bar s,\bar s]$. The claim then follows from \eqref{eq:claimH}.

From  \eqref{eq:hatppiconst} and  the definition of speed of a test plan (see \cite[Section 2.3.5]{Gigli14}) we have that
\[
0=f\circ \e_s-f\circ\e_t=\int_t^s[\e_r^*\d f](\hat\ppi'_r)\,\d r\qquad\hat\ppi-a.e.\ \forall t,s\in[0,1],\ t<s.
\]
By Fubini's theorem this implies that for a.e.\ $t\in[0,1]$ the identity $[\e_t^*\d f](\hat\ppi'_t)=0$ holds $\hat\ppi$-a.e.. Now we can use our assumption $\d g=h\d f$ (noticing that $|h|\in L^\infty_{loc}(\X)$ as a consequence of $|\d g|,\tfrac1{|\d f|}\in L^\infty_{loc}(\X)$) to deduce that
\[
g\circ \e_1-g\circ \e_0=\int_0^1[\e_r^*\d g](\hat\ppi'_r)\,\d r=\int_0^1 h\circ\e_r[\e_r^*\d f](\hat\ppi'_r)\,\d r=0\qquad\hat\ppi-a.e..
\]
This latter identity, however, is in contradiction with the fact that
\[
\int g\circ \e_0\,\d\hat\ppi=\int g\,\d(\e_0)_*\hat\ppi\leq \sup_{V_0}g<\inf_{V_1}g\leq \int g\,\d(\e_1)_*\hat\ppi=\int g\circ \e_1\,\d\hat\ppi,
\]
thus proving the claim \eqref{eq:claim1}.

Property  \eqref{eq:claim1} defines the real valued function $\varphi$ on the connected - being the continuous image of a connected - set $I:=f(B_{\bar r}(\bar x))\subset\R$. To conclude the proof it is therefore enough to show that $\varphi:I\to\R$ is locally Lipschitz, with a control on the local Lipschitz constant independent on the chosen neighbourhood.

Thus let $x\in B_{\bar r}(\bar x)$, put $\alpha:=f(x)$ and notice that for $|t|\ll 1$ we have $\Fl_t(x)\in B_{\bar r}(\bar x)$ as well with - by the above discussion - $f(\Fl_t(x))=\alpha+t$. Conclude noticing that
\[
\begin{split}
|\varphi(\alpha+t)-\varphi(\alpha)|&=|\varphi(f(\Fl_t(x)))-\varphi(f(x))|=|g(\Fl_t(x))-g(x)|\leq\Lip(g\restr{B_{\bar r}(\bar x)})\|v\|_{L^\infty}|t|,
\end{split}
\]
where in the last inequality we used the fact that the speed of the curve $s\mapsto \Fl_s(x)$ is uniformly bounded by $\|v\|_{L^\infty}$.
\end{proof}

In general this last result cannot be globabilized.  In our case, however, this is possible:
\begin{prop}\label{prop:odephi}
With the same notation and assumptions as in Assumption \ref{baseas}  the following holds.

There exists a  function $\varphi\in C^{\infty}_{loc}(I)$ such that $|\nabla u|=\varphi\circ u$, where $I:=u(\X)$, and it satisfies
\begin{equation}
\label{eq:odephi}
\tfrac{1}{1-N}\varphi\varphi''=1-\tfrac1{(1-N)^2}(\varphi')^2.
\end{equation}
\end{prop}
\begin{proof} For $x\in\X$ let $U_x,\varphi_x$ the neighbourhood and the function given by Proposition \ref{prop:comploc} above. By item $(2)$ in Proposition \ref{hessianbusemann} we know that $\varphi_x$ is strictly positive. Now suppose for a moment that we already know that $\varphi_x$ is $C^{1,1}_{loc}$ for every $x\in\X$.

We know that $|\d u|=\varphi_x\circ u$ holds on $U_x$, thus the regularity of $\varphi_x$ and its positivity justify the chain rules
\[
\begin{split}
\Delta(|\d u|^{\frac{N-2}{N-1}})&=\div(\tfrac{N-2}{N-1}(\varphi_x^{-\frac1{N-1}}\varphi_x')\circ u\nabla u)=\Big(-\tfrac{N-2}{(N-1)^2}\,\varphi_x^{\frac {N-2}{N-1}}(\varphi_x')^2 +\tfrac{N-2}{N-1}\,\varphi_x^{1+\frac{N-2}{N-1}}\varphi_x''\Big)\circ u
\end{split}
\]
having used also the fact that $u$ is harmonic.

Since \eqref{eq:rigdelta} can now be written as $\Delta(|\d u|^{\frac{N-2}{N-1}})=-(N-2)\varphi_x^{\frac{N-2}{N-1}}\circ u$ we conclude that $\varphi_x$ satisfies \eqref{eq:odephi}. It then follows by standard bootstrapping that $\varphi_x$ is $C^\infty_{loc}$, as in the statement.

The ODE \eqref{eq:odephi} also gives, by ODE uniqueness, the desired rigidity, as it is clear that any two solutions $\varphi_x,\varphi_{x'}$ defined on some intervals $I_x,I_{x'}\subset(0,1)$ that coincide in some non-trivial interval $I\subset I_x\cap I_{x'}$ must coincide on the whole $ I_x\cap I_{x'}$ and that the natural `glued' function defined on $ I_x\cup I_{x'}$ is still a solution. Thus fix $\bar x\in\X$, let $\varphi$ be the maximal solution of \eqref{eq:odephi} extending $\varphi_{\bar x}$ and let $A\subset \X$ the set of those $x$'s such that  $\varphi_x$ is the restriction of $\varphi$ to some subinterval of its domain of definition. Since $U_x$ is a neighbourhood of $x$ and $\varphi_x$ is defined on the open set $u(U_x)$ (by item $(5)$ in Proposition \ref{hessianbusemann}), it is clear that  $A$ is both open and closed. Since it is not empty and $\X$ is connected (being geodesic) we conclude that $A=\X$.

We now prove that  $\varphi_x$ is $C^{1,1}_{loc}$ and to this aim we start claiming that for $\eta\in W^{1,2}(\X)$ with support compact and contained in $U_x$ the measure  $\mu:=u_*(\eta\mm)$ is absolutely continuous w.r.t.\ $\mathcal L^1$ with density that is also absolutely continuous. Indeed, let $\zeta$ be a Lipschitz cut-off function with compact support and identically 1 on a neighbourhood of $\supp(\eta)$ and let $(\Fl_t)$ be the Regular Lagrangian Flow of $\zeta\nabla u$ (whose existence and uniqueness follow from $|\zeta\nabla u|\in L^\infty$, $\div(\zeta\nabla u)=\la\nabla\zeta,\nabla u\ra\in L^\infty$ and $\zeta\nabla u\in W^{1,2}_C(T\X)$ by \cite[Proposition 3.3.22]{Gigli14}). Then by \cite[Proposition 2.7]{GR17} we see that $\frac{u\circ\Fl_t-u}{t}\to |\d u|^2$ in $L^2(\supp(\eta),\mm)$ as $t\to0$. Notice also that item $(2)$ in Proposition \ref{hessianbusemann} yields  that $\tfrac1{|\d u|^2}\in W^{1,2}_{loc}\cap L^\infty_{loc}(\X)$ (for Sobolev regularity recall that $u\in H^{2,2}_{loc}(\X)$ and \cite[Proposition 3.3.22]{Gigli14}), thus the following computation is justified for any   $\xi\in C^\infty_c(\R)$:
\[
\begin{split}
-\int \xi'\,\d\mu&=-\int \xi'\circ u\,\eta\,\d\mm=
-\lim_{t\to 0}\int \frac{\xi\circ u\circ\Fl_t-\xi\circ u}t\,\frac{\eta}{|\d u|^2}\,\d\mm\\
&=\int \xi\circ u\, \langle \d\Big(\frac{\eta}{|\d u|^2}\Big),\d u\rangle\,\d\mm=\int \xi\,\d\nu,\qquad\text{where}\quad \nu:=u_*\Big( \langle \d\Big(\frac{\eta}{|\d u|^2}\Big),\d u\rangle\,\mm\Big).
\end{split}
\]
This proves that the distributional derivative of  $\mu$ is a Radon measure, and thus that   $\mu\ll\mathcal L^1$ with BV density. To prove that the density is absolutely continuous it suffices to prove that   $\nu\ll\mathcal L^1$. But this is obvious, as a direct consequence of what just proved is that   $u_*\mm\ll\mathcal L^1$.

We are now ready to show that  $\varphi=\varphi_x$ is $C^{1,1}_{loc}(u(U_x))$. We know from Proposition \ref{prop:comploc} that it is locally Lipschitz and that in $U_x$ we have $|\d u|=\varphi\circ u$. Since we already recalled that $|\d u|\in W^{1,2}_{loc}$ we see that  $\d|\d u|=\varphi'\circ u\,\d u$. Also, we know from \eqref{eq:rigdelta} and the chain rule for the Laplacian that in $U_x$ we have  $\Delta(|\d u|)=\psi\circ u$ for some locally bounded function $\psi$, that we can rewrite as $\divv(\varphi'\circ u\,\d u)=\psi\circ u$. Now we take $\eta\in \Test(\X)$  with support in  $U_x$ and $\xi\in C^\infty(\R)$ and observe that
\[
\begin{split}
\int\langle \d(\xi\circ u\,\eta),\varphi'\circ u\,\d u\rangle\,\d\mm =-\int(\psi \,\xi)\circ u\,\eta\,\d\mm=-\int \psi \,\xi \rho\,\d\mathcal L^1,
\end{split}
\]
where   $\rho$ is the density  of  $u_*(\eta\mm)$. On the other hand we also have
\[
\begin{split}
\int\langle \d(\xi\circ u\,\eta),\varphi'\circ u\,\d u\rangle\,\d\mm &=\int(\xi'\varphi')\circ u\,|\d u|^2\eta+(\varphi'\xi)\circ u\langle \d\eta,\d u\rangle\,\d\mm\\
&=\int \xi'\varphi'\rho_1+\xi\varphi'\rho_2\,\d\mathcal L^1,
\end{split}
\]
where  $\rho_1,\rho_2$ are the densities of  $u_*(|\d u|^2\eta\mm)$ and $u_*(\langle \d\eta,\d u\rangle\,\mm)$, respectively. By the arbitrariness of $\xi$ we proved that the distributional derivative of $\varphi'\rho_1$ is equal to  $\varphi'\rho_2+\psi \rho\in L^1_{loc}$, and thus that   $\varphi'\rho_1$ is (more precisely: has a representative that is) absolutely continuous. Since $|\d u|^2\eta$ is in $W^{1,2}(\X)$ and has support in $U_x$, what previously proved shows that $\rho_1$ is also absolutely continuous.

We thus deduce that $\varphi'$ is locally absolutely continuous on $\{\rho_1>0\}=u(\{\eta>0\})$. Taking $\eta=\eta_n$ for $(\eta_n)$ such that of $\cup_n\{\eta_n>0\}=U_x$ we conclude.
\end{proof}

\subsection{The splitting in our setting}
We are now ready to prove Theorem \ref{mainthm}, whose statement for convenience we repeat:
\begin{thm}
Let \((\X,\di,\m)\) be an \(\RCD(-(N-1),N)\) space with \(N\geq 3\) and assume that the first eigenvalue of the Laplacian \(\lambda_1\) is   \(\geq N-2\). Then one of the following holds:
\begin{enumerate}
	\item \(\X\) has exactly one end with infinite volume;
	\item \(\X\) is isomorphic to a warped product space \(\R\times_{w} \X'\) with warping functions $w_\sfd:=\cosh$, $w_\mm:=\cosh^{N-1}$ 
	where \(\X'\) is a compact \(\RCD(-(N-1),N)\) space. Moreover, in this case \(\lambda_1=N-2\).
\end{enumerate}
\end{thm}
\begin{proof} Since  $\lambda_1\geq N-2>0$, we know from Proposition \ref{prop:1end} that $\X$ has at least one end of infinite measure. Assume it has at least two of these. Then Corollary \ref{thm2ends} gives the existence of a non-constant bounded harmonic function $f$ on the whole $\X$ and then Proposition \ref{prop:odephi} that
\begin{equation}
\label{eq:dff}
|\d f|=\varphi\circ f
\end{equation}
for a positive smooth function $\varphi$ on $f(\X)$  that satisfies
\begin{equation}
\label{eq:odeagain}
\tfrac{1}{1-N}\varphi\varphi''=1-\tfrac1{(1-N)^2}(\varphi')^2.
\end{equation}
Then Proposition \ref{hessianbusemann} tells that $|\d f|\neq 0$ a.e., and that putting $\e_1:=\tfrac{\nabla f}{|\nabla f|}$ we have
\begin{equation}
\label{eq:zeta}
\Hess f=\zeta\circ f|\d f|\Id-N\zeta\circ f|\d f|\e_1\otimes\e_1 \qquad\text{for}\qquad \zeta=\tfrac{\varphi'}{1-N}.
\end{equation}
We can therefore apply Lemma \ref{le:coord}: let $\eta$ be so that  $\eta'=\tfrac1\varphi$ ($\eta$ is defined up to an additive constant: the value of such constant will be chosen in a moment) and define the function $\b:=\eta\circ f$. Then $\b\in H^{2,2}_{loc}(\X)$ with
\begin{align}
|\d \b|&\equiv1,&&\\
\Hess \b&=\psi_\sfd\circ\b\big(\Id-\e_1\otimes\e_1\big),&&\text{for }\quad \psi_\sfd:=\zeta\circ\eta^{-1}\\
\Delta \b&=\psi_\mm\circ\b,&&\text{for }\quad\psi_\mm:=(N-1)\zeta\circ\eta^{-1}=(N-1)\psi_\sfd.
\end{align}
To find explicitly $\psi_\sfd,\psi_\mm$ notice that
\begin{equation}
\label{eq:conto}
\zeta'\varphi\stackrel{\eqref{eq:zeta}}=\tfrac{\varphi''\varphi}{1-N}\stackrel{\eqref{eq:odeagain}}=1-\tfrac{(\varphi')^1}{(N-1)^2}\stackrel{\eqref{eq:zeta}}=1-\zeta^2
\end{equation}
and thus
\[
\psi_\sfd'=(\zeta\circ\eta^{-1})'=(\zeta'\tfrac1{\eta'})\circ\eta^{-1}=(\zeta'\varphi)\circ\eta^{-1}\stackrel{\eqref{eq:conto}}=1-\zeta^2\circ\eta^{-1}=1-\psi_\sfd^2
\]
proving that $\psi_\sfd(z)=\tanh(z+c)$ for some $c\in\R$. Since $\psi_\sfd\circ\eta$ is equal to the given function $\zeta$, we see that replacing $\eta$ with $\tilde\eta:=\eta+\alpha$ for $\alpha\in\R$ means replacing $\psi_\sfd$ with $\tilde\psi_\sfd=\psi_\sfd(\cdot-\alpha)$. Thus we can, and will, choose $\eta$ so that $\psi_\sfd=\tanh$. 

It follows by the defining properties \eqref{eq:defwd}, \eqref{eq:normw} that $w_\sfd=\cosh$. Similarly, we have  $\psi_\mm =(N-1)\tanh $ and thus $w_\mm=\cosh^{N-1}$.

Then Theorem \ref{thm:split} gives the required warped product structure, where the fiber is the space $\X':=\b^{-1}(0)$ equipped with the distance $\sfd'$ and the measure $\mm'$ defined in \eqref{eq:defdp} and \eqref{eq:defmp} respectively. Moreover, we can apply Proposition \ref{prop:rcdxprimo} with $\bar z=0$ to deduce that $(\X',\sfd',\mm')$ is $\RCD(-(N-1),N)$.

It remains to prove that $\X'$ is compact. Say not. We are going to prove that in this case $\X$ has at most one end, thus contradicting the assumption made at the beginning of the proof.

Let $K\subset\X$ be compact and let $K':=\Pr^{-1}(\Pr(K))\cap \b^{-1}(\b(K))$. Since $\sfT(K')=\b(K)\times \Pr(K)\subset \R\times_w\X'$ we see that  the `rectangle' $K'$ is also compact. Let $x_0,x_1\in\X\setminus K'$. Then for $i=0,1$ either $\Pr(x_i)\notin \Pr(K')$ or $\b(x_i)\notin \b(K')$ (or both). Say $\b(x_i)\notin \b(K')$ for $i=0,1$ and use the assumption that $\X'$ is not compact to find $z\in \X'\setminus \Pr(K')$. Then the curve $t\mapsto\Fl_t(z)$ (here $\Fl$ is the flow of $\nabla\b$) does not meet $K'$ and, moreover, there are $t_0,t_1\in\R$ with $\b(\Fl_{t_i}(z))=\b(x_i)$, $i=0,1$. Recalling that the level sets of $\b$ are path connected, we can find curves joining $x_i$ and $\Fl_{t_i}(z)$ lying entirely on level sets. The assumption $\b(x_i)\notin \b(K')$  ensures that these curves do not meet $K'$, thus the path obtained by gluing these curves and $t\mapsto\Fl_t(z)$  produces a curve from $x_0$ to $x_1$ that does not meet $K'$, showing that $x_0,x_1$ belong to the same connected component of $\X\setminus K'$. Since an analogous construction can be made if one, or both, of the  $x_i$'s are with $\Pr(x_i)\notin \Pr(K')$, we see that $\X\setminus K'$ has only one connected component, providing the desired contradiction.
\end{proof}

\def\cprime{$'$} \def\cprime{$'$}

\end{document}